\documentclass[10pt]{article}

\usepackage{graphicx} 

\title{A Mixed-FEM approximation with uniform conservation of the exponential stability for a class of anisotropic port-Hamiltonian system and its application to LQ control.}
\author{Luis Mora, Kirsten Morris.}
\date{May 2024}

\usepackage{amsfonts, amsthm, amsmath, amssymb, comment, color}
\usepackage[margin=2cm]{geometry}
\usepackage{dutchcal}
\usepackage{cite}

\newcommand{\T}{\top}
\newcommand{\disp}{\displaystyle}

\newcommand{\inner}[1]{\left\langle #1 \right\rangle}
\newcommand{\x}{x}

\newcommand{\pd}[2]{\frac{\partial #1}{\partial #2}}
\renewcommand{\L}{\Theta}

\newcommand{\R}{\mathbb{R}}
\newcommand{\eig}{ {\rm eig}}
\renewcommand{\d}{{\rm d}}

\newcommand{\qv}{\boldsymbol{q}}
\newcommand{\pv}{\boldsymbol{p}}
\newcommand{\ev}{\boldsymbol{e}}
\newcommand{\bp}{\theta}
\newcommand{\bpi}{\overline{\bp}}

\newcommand{\ap}{a}
\newcommand{\api}{\overline{\ap}}
\newcommand{\uv}{\boldsymbol{u}}

\newcommand{\Ha}{\mathcal{H}}

\newcommand{\tr}{\boldsymbol{t}_r}
\newcommand{\xx}{\mathrm{z}}
\newcommand{\xl}{x_l}
\newcommand{\xr}{x_r}

\newcommand{\ub}{\mathbf{u}}
\newcommand{\vb}{\mathbf{v}}

\newcommand{\M}{M}
\newcommand{\D}{D}
\newcommand{\C}{C}
\newcommand{\A}{A}
\newcommand{\W}{W}
\newcommand{\B}{B}

\newcommand{\Lh}{\bar{\L}}
\newcommand{\Mh}{\bar{\M}}
\newcommand{\Dh}{\bar{\D}}
\newcommand{\Bh}{\bar{\B}}

\newcommand{\Ch}{\bar{\C}}
\newcommand{\Wh}{\bar{\W}}
\newcommand{\Ah}{\bar{\A}}

\newcommand{\dirac}[1]{\boldsymbol{\delta}\left(\x - #1\right)}

\newtheorem{lemma}{Lemma}
\newtheorem{definition}{Definition}
\newtheorem{assump}{Assumption}
\newtheorem{remark}{Remark}
\newtheorem{theorem}{Theorem}
\newtheorem{prop}{Proposition}

\begin{document}

\maketitle

\begin{abstract}
	In this manuscript, we present a mixed finite element discretization for a class of boundary-damped anisotropic port-Hamiltonian systems. Using a multiplier method, we demonstrate that the resulting approximation model uniformly preserves the exponential stability of the uncontrolled system, establishing a lower bound for the exponential decay rate that is independent of the mesh size. This property is illustrated through the spatial discretization of a piezoelectric beam. Furthermore, we show how the uniform preservation of exponential stability by the proposed model aids in the convergence of controllers derived from an infinite-time linear quadratic control design, in comparison to models obtained from the standard finite-element method. 
\end{abstract}

\section{Introduction.}

 The conservation of the exponential stability independently of the mesh size, uniform conservation, is a desired property for the numerical approximation models of distributed-parameter systems. The relevance of this property lies not only in mitigating the effects of spurious oscillations during simulations, as happen in standard finite differences approximations of the wave equation, but also in the conservation of properties as the exact controllability and observability´that are relevant in the design of controllers and estimators \cite{Morris2020}. For example, in the linear quadratic control for partial differential equations with an infinite time horizon, the design is based on numerical approximations of the systems, and the sequence of designed controllers could not converge as the approximation order is increased and may not even stabilize the original system. The key sufficient condition for the controllers convergence and the satisfactory performance of the feedback system is the uniform exponential stability of the numerical model \cite{Morris1994,Morris2020}.
 
 Unfortunately, not all discretization schemes can conserve uniformly the exponential stability of any distributed-parameter system. For example, standard finite-differences (FD) and finite-elements (FE) methods does not conserve the stability properties of the boundary-damped wave equation, as shown in \cite{Banks1991,Delaunay2024}. In the last 30 years, several modifications of the finite-differences and finite-elements methods have been proposed to guarantee a uniform exponential stability conservation. Approaches based on numerical viscosity terms are proposed in \cite{Delaunay2024,Munch2007,Ramdani2007,Zuazua2012}, based on order-reduction methods in \cite{Liu2020,Ren2022,Guo2020,Liu2022,Wang2023}, and using average operators in  \cite{Ozer2024,Zhang2024}. Additionally, approaches based on mixed finite-elements with uniform exponential decay can be found in \cite{Banks1991,Abdallah2013,Egger2019,Egger2018,Egger2024,DelReyFernandez2023}, modal methods in \cite{Liu1994}, and continuous Galerkin in \cite{Fabiano2001}.
 
  On the other hand, the port-Hamiltonian systems is a framework focused on the system energy flux and the description complex interconnected processes, including those characterized by ordinary and partial differential equations, linear or nonlinear, \cite{Jacob2012,LeGorrec2005,vdS14}. Between the advantages of this framework is that the well-posedness of boundary-controlled distributed-parameter systems can be done by a simple rank analyzing of a set of matrices that define the boundary variable relations \cite{Zwart2010,Jacob2012}. Similarly, these matrices are used to study the system exponential stability, see \cite{Augner2014,Villegas2009a,Trostorff2022} for details, however, these papers do not provide the decay rate. In \cite{Mora2023} the multiplier method was used to obtain a lower bound on the exponential decay rate in terms of the physical parameters for a class anisotropic port-Hamiltonian systems  with boundary dissipation.
 The discretization schemes of distributed-parameter port-Hamiltonian systems center on the conservation of the port-Hamiltonian structure, i.e., the finite-dimensional approximation model is also port-Hamiltonian and conserves the energy flux structure between the system components. In the literature we can find schemes based on finite-differences \cite{Trenchant2018}, finite volumes \cite{Kotyczka2016}, partitioned finite-elements \cite{Serhani2019a,Serhani2019b,Haine2023}, mixed finite-elements \cite{DelReyFernandez2023,Golo2004,Wu2015}, discontinuous Galerkin methods \cite{Thoma2023}, and pseudo-spectral methods \cite{Moulla2012}. The uniform stability conservation by these discretization schemes is commonly studied evaluating numerically the eigenvalues of the obtained finite-dimensional model \cite{Moulla2012, Thoma2023, Serhani2019b}.
 
 Other challenge for the analysis of the approximation schemes is the choose of  the method used to study the uniform conservation of the exponential stability. For example, the eigenvalue method can be used to obtain an analytic expression of the boundary-damped wave equation \cite{Cox1995}, however, for the discretized model, obtaining an analytic formula for the eigenvalues of generic matrices with arbitrary size is a complex task. Alternative approaches based on a frequency analysis were used in \cite{Liu1994,Wang2023}, among others. The multiplier methods are a powerful tool to obtain an explicit bound for the exponential decay rate of ordinary and partial differential equations \cite{Tucsnak2009,Komornik1994} and their numerical approximations \cite{Egger2024,Zhang2024,Ozer2024,Guo2020}.
 
 In this work, we present a mixed finite-element method (MFEM) the spatial discretization of a class of anisotropic distributed-parameters port-Hamiltonian system with uniform conservation of the exponential stability properties and its application on the linear quadratic control design. To preserves the port-Hamiltonian structure, we express the dynamical system in terms of the co-energy variables \cite{vdS14,Serhani2019a}. Then, the discretization scheme using linear piecewise bases and constant piecewise test functions is applied. For the exponential stability analysis, we use the multiplier method described in \cite{Tucsnak2009}. This method is useful to obtain an explicit expression of the exponential decay for distributed-parameter port-Hamiltonian systems \cite{Mora2023} and numerical approximations \cite{DelReyFernandez2023}. We use a piezoelectric beam model as example of the uniform conservation of the exponential stability analysis. Finally, using the wave equation, we show how the control law obtained from a linear quadratic controller design based on the MFEM model proposed, converges to a control law expressed in terms of the continuous system variables. 

\section{Distributed-parameters systems anisotropic port-Hamiltonian systems.}\label{sec:PHS}

This work considers a class of anisotropic distributed-parameters port-Hamiltonian systems on interval $[\xl,\xr]$. We separate the state variables into two vectors of size $n\in\mathbb{N}$,  $\qv(\x,t)=\begin{bmatrix}
    q_1(\x,t),\dots, q_n(\x,t)
\end{bmatrix}^\top \in H^{1}(\xl,\xr;\mathbb{R}^{n})$ and $\pv(\x,t)=\begin{bmatrix}
    p_1(\x,t),\dots, p_n(\x,t)
\end{bmatrix}^\top \in H^{1}(\xl,\xr;\mathbb{R}^{n})$. Similarly, the space-varying physical parameters are grouped in two sets $\{\bp_i^{q}(\x)\}_{i=1}^{n}$ and $\{\bp_i^{p}(\x)\}_{i=1}^{n} $, strictly positive on $C^1 (\xl,\xr)$.  Then, the system dynamics are expressed as:
\begin{align}
	&\pd{}{t} \begin{bmatrix}
        \qv(\x,t)\\ \pv(\x,t)
    \end{bmatrix}-\underset{P_1}{\underbrace{\begin{bmatrix}
        0 & A\\A^\top & 0
    \end{bmatrix}}}\pd{ }{\x} \left(\underset{\L(\x)}{\underbrace{\begin{bmatrix}
        \L_{q}(\x) & 0\\ 0 & \L_p(\x)
    \end{bmatrix}}}  \begin{bmatrix}
        \qv(\x,t)\\ \pv(\x,t)
    \end{bmatrix}\right)=\begin{bmatrix}
		    B_q(\x) \\ B_p (\x)
		\end{bmatrix} \uv(t), \quad \x\in[\xl,\xr], \label{eq:PHS_2}\\
  &A^\top\L_q(\xr)\qv(\xr,t)=-K\L_p(\xr)\pv(\xr,t),\label{eq:BC_b} \\
		&\L_p(\xl)\pv(\xl,t) =0,
        \label{eq:BC_a}
\end{align}
where $\ub(t)\in\mathbb{R}^l$ denotes control signal, $A \in \mathbb{R}^{n \times n}$ is invertible, $\L_q (\x)=diag\left(\bp^{q}_1(\x),\dots, \bp^{q}_{n}(\x)\right)$, $\L_p (\x)=diag\left(\bp^{p}_1(\x),\dots, \bp^{p}_{n}(\x)\right)$, $B_q(\x),B_p(\x) \in \mathcal{L}\left(\R^{l},L^2(\xl,\xr;\R^n)\right)$, $K=K^\top >0 \in\mathbb{R}^{n \times n}$ is the boundary dissipation matrix. Letting $\inner{\cdot, \cdot}_{L^2}$ indicate the usual inner product on $L^2(\xl,\xr),$ the system energy is
\begin{align}
    \Ha(t)=\frac{1}{2}\inner{\begin{bmatrix}
        \qv(\x,t)\\ \pv(\x,t)
    \end{bmatrix},\L(\x) \begin{bmatrix}
        \qv(\x,t)\\ \pv(\x,t)
    \end{bmatrix}}_{L^2}.
\end{align}
It is straightforward to verify that 
\begin{align}
    \dot{\Ha}(t)=&- \pv^\top(\xr,t)\L_p(\xr)K\L_p(\xr)\pv(\xr,t) +\inner{\L(\x) \begin{bmatrix}
        \qv(\x,t)\\ \pv(\x,t)
    \end{bmatrix}, \begin{bmatrix}
		    B_q(\x) \\ B_p (\x)
		\end{bmatrix} \uv(t)}_{L^2}.
\end{align}
and so the system is dissipative; that is, without external stimulus the energy is non-increasing along trajectories. 
The exponential stability of this system without control, $\uv(t)=0$ can be analyzed using the approach in \cite{Mora2023}, if the physical parameters satisfy the following assumption.

\begin{assump}\label{assump:1}
    Defining $m(\x)=\x-\xl , $ there is a constant $\delta^c>0$ such that for every $i=1\ldots n,\; \xl \leq x \leq \xr ,$
    \begin{equation}
    \begin{array}{ll}
        \theta^p_i(\x)-m(\x)(\theta_i^p)^\prime(\x)& >\delta^c \theta^p_i(\x), 
        \\
          \theta^q_i(\x)-m(\x)(\theta_i^q)^\prime(\x)& >\delta^c \theta^q_i(\x).
    \end{array}
    \label{eq:1A}
      \end{equation}
\end{assump}

Then, a lower bound for the exponential decay rate $\alpha$ can be obtained using the next Theorem.
\begin{theorem}
 \cite[Theorem 1]{Mora2023} Consider $\uv(t)=0$. If Assumption \ref{assump:1} holds, 
 then, defining 
\begin{align}
 \epsilon_0=\dfrac{\eta_{\L}}{\ell \mu_{P_1}},\quad    \epsilon_1=\dfrac{2\eta_K}{\ell \mu_\Psi}, \label{eq:epsilonc} 
\end{align}
and with $\mu_{\Psi}= \max eig(\Psi)$, $\Psi=\begin{bmatrix}
    A^{-\top} K\\I_n
\end{bmatrix}^\top \L^{-1}(\xr) \begin{bmatrix}
    A^{-\top} K\\I_n
\end{bmatrix}$, $\eta_K=\min eig(K)$, $\disp\eta_\L=\min_{x\in[\xl,\xr]} eig(\L(\x)) $, and $\mu_{P_1}=\sqrt{\max eig(P_1^{-2})},$ 
 system \eqref{eq:PHS_2}--\eqref{eq:BC_a} is exponentially stable with a decay rate bounded by
\begin{align*}
    \alpha=\dfrac{\delta^c \epsilon \epsilon_0}{\epsilon+\epsilon_0} , \quad\forall \epsilon\in (0,\min \{\epsilon_0,\epsilon_1\}].   
\end{align*}
\end{theorem} 

To reduce repetition, the notation $\xx$, where  $\xx=p$ or $q$, will be used in definitions and calculations applicable to either states $\pv$ or $\qv$ of system \eqref{eq:PHS_2}.

For a system with anisotropic physical parameters, a direct discretization of the state variables $\pv,\qv$ does not necessarily lead to a port-Hamiltonian structure in the numerical approximation. To solve this issue, several researchers reformulate the system using the co-energy variables, also called ``efforts", as state variables 
\cite{Serhani2019b,Serhani2019a}.
The co-energy variable $e^{\xx}_i$ is defined as the variational derivative of $\Ha(t)$ with respect to the state $\xx_i$, that is, $e^{\xx}_i(\x,t)=\delta_{\xx_i}\Ha(t)$. For system \eqref{eq:PHS_2}--\eqref{eq:BC_a} we obtain $e_{i}^{\xx}(\x,t)=\bp_{i}^{\xx}(\x)\xx_i(\x,t)$, that is, the co-energy vector corresponding to state $\xx(\x,t)$ is $\ev^{\xx}(\x,t)=\begin{bmatrix}
       e^{\xx}_1(\x,t) & \dots & e^{\xx}_n(\x,t) 
   \end{bmatrix}^\top=\L_{\xx}(\x)\xx(\x,t)$. Consequently, system \eqref{eq:PHS_2}--\eqref{eq:BC_a} is rewritten as
   \begin{align}
               \L^{-1}(\x)\pd{}{t}\begin{bmatrix}
           \ev^{q}(\x,t) \\ \ev^p(\x,t)
       \end{bmatrix}=&P_1 \pd{}{\x} \begin{bmatrix}
           \ev^{q}(\x,t) \\ \ev^p(\x,t)
       \end{bmatrix}+\begin{bmatrix}
		    B_q(\x) \\ B_p (\x)
		\end{bmatrix} \uv(t), \quad \forall\x\in[\xl,\xr],\label{eq:PHS_3}\\
       A^\top \ev^q(\xr,t)=&-K\ev^p(\xr,t),\label{eq:BC_b2}\\
       \ev^p(\xl,t)=&0,\label{eq:BC_a2}
   \end{align}
with energy
\begin{align}
    \Ha(t)=\frac{1}{2}\inner{\begin{bmatrix}
        \ev^q(\x,t)\\ \ev^p(\x,t)
    \end{bmatrix}, \L^{-1}(\x) \begin{bmatrix}
        \ev^q(\x,t)\\ \ev^p(\x,t)
    \end{bmatrix}}_{L^2},
\end{align}
satisfying
\begin{align}
    \dot{\Ha}(t)=&-[\ev^p(\xr,t)]^\top K \ev^p(\xr,t)+\inner{ \begin{bmatrix}
        \ev^q(\x,t)\\ \ev^p(\x,t)
    \end{bmatrix}, \begin{bmatrix}
		    B_q(\x) \\ B_p (\x)
		\end{bmatrix} \uv(t)}_{L^2}. \label{eq:dHc}
\end{align}

\section{ Mixed-FEM numerical model.}\label{sec:MFEM}

 \begin{figure}
	\centering
	\includegraphics[width=0.65\columnwidth]{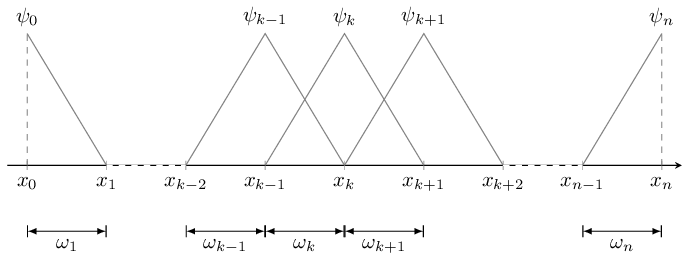}
	\caption{Mixed finite elements basis and test functions}
	\label{fig:my_label}
\end{figure} 

Consider a uniform partition of the spatial domain $[\xl,\xr]$ into $N$ elements of length $h=(\xr-\xl)/N=\ell/N$ and nodes $x_k=\xl+kh$ with $k\in\{0,\dots,N\}$.  Define the piecewise linear basis $\psi_k(\x)$ (see Figure \ref{fig:my_label})
\begin{align}
    \psi_k (x) = \begin{cases}
	\dfrac{x-x_{k-1}}{h}, & x_{k-1}  \leq x < x_{k} \\
	\dfrac{x_{k+1} -x}{h} , & x_{k}  \leq x \leq x_{k+1} \\ 
	0, & {\rm otherwise.}   
\end{cases}, \quad \forall\x\in[\xl,\xr] \text{ and } k\in\{0,\dots,N\}  ,\label{eq:basis}
\end{align}
and the piecewise constant test function $\omega_k$
\begin{align}
    \omega_k (\x) = \begin{cases}
	\frac{1}{h}, & \x_{k-1}  \leq \x < \x_{k} \\ 
	0, & {\rm otherwise}  
\end{cases},  \quad \forall\x\in[\xl,\xr] \text{ and } k\in\{1,\dots,N\}  \label{eq:test}
\end{align}
 we obtain the following properties.

\begin{lemma} \label{lem:1}
Let $\psi_k(\x)$ and $\omega_k(\x)$ be the base and test functions defined at \eqref{eq:basis} and \eqref{eq:test}, respectively, and $\nu_k(\x)=\disp\left(\sum_{l=1}^{N}a_l h \omega_l(\x)\right)\psi_k(\x)$ with $a_l\in\mathbb{R}$. The following statement holds
     \begin{itemize}
     \item[\rm (a)] For any $i,j\in \{1,\dots, N\}$,  $h\omega_i(\x)\omega_j(\x)=\omega_i(\x)\delta_{ij}$ and $h\omega_i(\x)\psi_N(\x)=\psi_N(\x)\delta_{iN}$ for all $\x\in[\xl,\xr]$, where $\delta_{ij}$ denotes the Kronecker delta.

     \item[\rm (b)] $\displaystyle \inner{ \omega_j(\x),\psi_k(\x)}$ is equal to $\dfrac{1}{2}$ if $j\in\{k,k+1\}$ and $0$ otherwise.  

    \item[\rm (c)] $\inner{\omega_j(\x),\psi'_k(\x)}$ is equal to $-\dfrac{1}{h}$ if $j=k+1$, $\dfrac{1}{h}$ if $j=k$, and $0$ otherwise. 

     \item[\rm (d)]  $\inner{\omega_j(\x), \nu_k(\x)}$ is equal to $\dfrac{a_j}{2}$ if $j\in\{k,k+1\}$ and $0$ otherwise. 
    
    \item[\rm (e)] $\inner{\omega_j(\x),\nu'_k(\x)}$ is equal to $- \dfrac{a_k}{h}$ if $j=k+1$, $ \dfrac{a_k}{h}$ if $j=k$, and $0$ otherwise.

 \end{itemize}

\end{lemma}

\begin{proof}
    (a) is obtained directly from the fact that $\omega_j(\x)\omega_k(\x)=0$ for any $j\neq k$ and $\omega_j(\x)\psi_N(\x)=0$ for every $j\neq N$. (b) and (c) are obtained by solving the integrals. (d) is a direct consequence of (a) and (b). Finally, (e) is obtained as follows: Let $\dirac{\x_k}$ be the Dirac delta function centered at $\x_k$, satisfying $f(\x)\dirac{\x_k}=f(\x_k)\dirac{\x_k}$ and $\displaystyle\int_{\xl}^{\xr}\dirac{\x_k} \d\x=1$, for any $\x_k\in[\xl,\xr]$, then, $[h\omega_k(\x)]'=\dirac{\x_{k-1}}-\dirac{\x_k}$. Moreover, since $\omega_l(\x)\psi_k(\x)=0$ for every $l\notin \{k,k+1\}$, we get $\nu_k(\x)=(a_kh\omega_k(\x)+a_{k+1}h\omega_{k+1}(\x))\psi_{k}(\x)$. As a consequence,
    \begin{align*}
        \nu'_k(\x)=&[a_kh\omega_k(\x)+a_{k+1}h\omega_{k+1}(\x)]'\psi_k(\x)+[a_kh\omega_k(\x)+a_{k+1}h\omega_{k+1}(\x)]\psi'_k(\x),\\
        =&[a_k\dirac{\x_{k-1}}+(a_{k+1}-a_k)\dirac{\x_k} -a_{k+1}\dirac{\x_{k+1}}]\psi_k(\x)+[a_kh\omega_k(\x)+a_{k+1}h\omega_{k+1}(\x)]\psi_k'(\x).
    \end{align*}
     
     Since $\psi_k(\x_{k-1})=\psi_{k}(\x_{k+1})=0$ and $\psi_k(\x_k)=1$, using  (a) and the Dirac delta function properties, 
    \begin{align*}
        \inner{\omega_j(\x),\nu'_k(\x)}=& \inner{ \omega_j(\x), [a_k\dirac{\x_{k-1}}+(a_{k+1}-a_k)\dirac{\x_k} -a_{k+1}\dirac{\x_{k+1}}]\psi_k(\x)}\\
        &+\inner{ [a_kh\omega_j(\x)\omega_k(\x)+a_{k+1}h\omega_j(\x)\omega_{k+1}(\x)],\psi_k'(\x) }\\
        =&(a_{k+1}-a_k) \omega_j(\x_k) + \inner{ [a_k\omega_j(\x)\delta_{jk}+a_{k+1}\omega_j(\x) \delta_{j(k+1)}],\psi_k'(\x)},
    \end{align*}

    From \eqref{eq:test}  we have that $\omega_j(\x_k)$ is equal to $\dfrac{1}{h}$, when $j= k+1$, and $0$ otherwise. Then, (c) lead us to
    \begin{align*}
        \inner{\omega_j(\x),\nu'_k(\x)}=&0, \forall j\notin \{k,k+1\}\\
        \inner{\omega_k(\x),\nu'_k(\x)}=&a_k\inner{\omega_k(\x),\psi_k'(\x) }=\dfrac{a_k}{h}\\
        \inner{\omega_{k+1}(\x),\nu'_k(\x)}=&\dfrac{a_{k+1}-a_k}{h} + a_{k+1}\inner{\omega_{k+1}(\x),\psi_k'(\x)} =-\dfrac{a_k}{h},
    \end{align*}
    completing the proof.
\end{proof}

Using the test functions $\psi_k(\x)$ and bases $\omega_k(\x)$, we define the following approximations of the co-energy variables and physical parameters.
\begin{definition}\label{def:approx}
    Consider $\xx\in\{p,q\}$ and $i\in\{1,\dots,n\}$. Let $\bpi_{i}^{\xx}$(\x) be the inverse of physical parameter $\bp_i^{\xx}(\x)$.
   The approximations of any co-energy variable $e_i^{\xx}$ and physical parameter inverse $\bpi_i^{\xx}$ are defined as
    \begin{align*}
     e_i^{\xx}(\x,t)\approx e^{\xx N}_i(\x,t)=&
        \sum_{k=0}^{N} e^{\xx}_{i,k}(t)\psi_k(\x) \quad \text{and} \quad
    \bpi_i^{\xx}(\x)\approx\bpi_{i}^{\xx N}(\x)=
        \sum_{k=1}^{N} h \bpi^{\xx}_{i,k}\omega_k(\x),
\end{align*}
respectively, where $e_{i,k}^{\xx}(t)=e_i^{\xx}(\x_k,t)$ and $\bpi^{\xx}_{i,k}=\bpi_i^{\xx}(\x_k)=\dfrac{1}{\bp_{i}^{\xx}(\x_k)}$ denote the co-energy variable and physical parameter inverse at node $\x_k$.    
\end{definition}

To impose the boundary conditions, consider the diagonal dissipation matrix $K=diag\left(\kappa_1,\dots,\kappa_n\right)$,
and the $ij$-th element of matrices $A$ and  $A^{-1}$ by $[A]_{ij}=\ap_{ij}$ and $[A^{-1}]_{ij}=\api_{ij}$.
  respectively. Then, from  the $i$-th row of \eqref{eq:BC_b2}, we obtain $\disp e^q_{i,N}(t)= -\sum_{l=1}^n\api_{li}\kappa_l e^p_{l,N}(t)$ or  $\disp \sum_{l=1}^n \ap_{li} e^q_{l,N}(t)=-\kappa_i e^p_{i,N}(t)$ and from \eqref{eq:BC_a2}, $e_{i,0}^p(t)=0$ for every $i\in\{1,\dots,N\}$. Consequently, approximations $e^{pN}_i(\x,t)$ and $e^{qN}_i(\x,t)$ are expressed as
\begin{align}
    e^{pN}_i(\x,t)=&\sum_{k=1}^{N}e^p_{i,k}(t)\psi_k(\x),\quad \text{and} \quad
    e^{qN}_i(\x,t)=
        \sum_{k=0}^{N-1}e^q_{i,k}(t)\psi_k(\x) -\sum_{l=1}^n\api_{li}\kappa_l e^p_{l,N}(t). \label{eq:BCeq3}
\end{align}
Similarly,
\begin{align}              
    \sum_{j=1}^{n}\ap_{ji}e^{qN}_j(\x,t)
        =&\sum_{j=1}^{n}\ap_{ji}\left(\sum_{k=0}^{N-1}e^q_{j,k}(t)\psi_k(\x)\right) -\kappa_i e^p_{i,N}(t) \psi_N(\x),\label{eq:BCeq1}
\end{align}
 or equivalently, 
\begin{align}              
    \sum_{j=1}^{n}\ap_{ji}\pd{e^{qN}_j(\x,t)}{\x}
        =&\sum_{j=1}^{n}\ap_{ji}\left(\sum_{k=0}^{N-1}e^q_{j,k}(t)\psi'_k(\x)\right) -\kappa_i e^p_{i,N}(t) \psi'_N(\x).\label{eq:BCeq2}
\end{align}

The following matrices and vectors will be used in   the numerical approximation.

\begin{definition}\label{def:hat_QDAMB}
   
    Let $I_N$ be the identity matrix of size $N \times N$. $L$,  $M$ and $D$  be $N\times N$ real matrices and $\tr$ be a vector of size $N$ defined as
    \begin{align}
       L=&\begin{bmatrix}
            0 &  & \\
            1 & \ddots&\\
              & \ddots & \ddots \\
             & &  1  & 0
        \end{bmatrix},\quad M=\dfrac{1}{2} (I_{N}+L^\top) ,\quad
        D=I_{N}-L, \quad \text{and} \quad \tr=\begin{bmatrix}
            0\\\vdots\\0\\1
        \end{bmatrix},\label{eq:MDtr}
    \end{align}
    respectively. For any $\xx\in\{p,q\}$ and $i\in\{1,\dots, n\}$,  matrix $\L_{\xx i}\in\R^{N\times N}$, the corresponding inverse $\L^{-1}_{\xx i}$, and $B_{\xx i}\in\R^{N\times l}$ are defined as
    \begin{align}
        \L_{\xx i}=diag\left(\bp^\xx_{i,1},\dots, \bp^\xx_{i,N} \right),  \quad  \L^{-1}_{\xx i}=diag\left(\bpi^\xx_{i,1},\dots, \bpi^\xx_{i,N} \right), \quad \text{and} \quad [B_{\xx i}]_{j,k}=\inner{\omega_j, [B_\xx]_{i,k}}. \label{eq:Qxi}
    \end{align} 
    Then, we build the following global matrices: 
\begin{align}
\begin{split}
    \Lh_\xx=&diag(\L_{\xx 1},\dots, \L_{\xx n}), \,
    \Ah=A\otimes I_{N}, \,
    \Dh=I_n \otimes D,\,
    \Mh= I_n\otimes M,   \\
    \Bh_2=&K\otimes \tr\tr^\top, \,
    \Bh_1=\dfrac{1}{2}\Ah^{-\top} \Bh_2, \, \Bh_{\xx}^\T=[B_{\xx 1}^\T,\ldots,B_{\xx n}^\T].
\end{split}
\end{align}
where $A\otimes B$ denotes the Kronecker product between matrices  $A\in\mathbb{R}^{n\times m}$ and $B\in\mathbb{R}^{l\times k}$, defined as
\begin{align*}
    A\otimes B=&\begin{bmatrix}
        a_{11}B & a_{12}B & \cdots & a_{1m}B\\
        a_{21}B & a_{22} B & &\vdots\\
         \vdots&&\ddots&\vdots\\
        a_{n1}B &\cdots&\cdots& a_{nm}B 
    \end{bmatrix}.
\end{align*}
\end{definition}

\begin{remark}
    Using the Kronecker product properties, see \cite[Proposition 9.1.6]{Bernstein2018}, we obtain the following alternative expressions for $\Bh_1$ and $\Bh_2$:
    \begin{align}
        \Bh_1=& \dfrac{1}{2}(I_n\otimes \tr)A^{-\top} K (I_n\otimes \tr^\top), \text{ and }
        \Bh_2= (I_n\otimes \tr)K(I_n\otimes \tr^\top).
    \end{align}
\end{remark}

Consider the vectors $\ev_i^{p}(t)=[e^{p}_{i,1}(t),\dots,e^{p}_{i,N}(t)]^\top$ and $\ev_i^{q}(t)=[e^{q}_{i,0}(t),\ldots,e^{q}_{i,N-1}(t)]^\top$ for any $i\in\{1,\ldots,n\}$. Using the properties of Lemma \ref{lem:1} and equations \eqref{eq:BCeq3}--\eqref{eq:BCeq2}, the inner product
\begin{align*}
    \inner{\omega_j(\x),\bpi_i^{qN}(\x)\pd{e^{qN}_i(\x,t)}{t}}_{L^2}=&  \inner{\omega_j(\x),\sum_{l=1}^{n}\ap_{il}\pd{e^{pN}_l(\x,t)}{\x}}_{L^2}+ \inner{\omega_j(\x),[B_q(\x)]_{i,-}\uv(t)}_{L^2}, 
\end{align*}
for every $j\in\{1,\dots, N\}$, where $[B_q]_{i,-}$ denotes the $i$-th row of $B_q$, lead us to the following approximation of the dynamics for $e_i^q$, 
 \begin{align}
      \L_{qi}^{-1} \left(M\dot{\ev}^q_i(t)-\disp\sum_{l=1}^{n} \dfrac{\api_{li}\kappa_l}{2} \tr\tr^\top\dot{\ev}^p_l(t)\right)=& \sum_{l=1}^n \ap_{il}\dfrac{1}{h}D\ev^p_i(t) + B_{qi} \uv(t).\label{eq:aux0}
 \end{align}
 
Similarly, to approximate the dynamics of $e^{p}_i$ we consider the inner product
\begin{align*}
    \inner{\omega_j(\x),\bpi^{pN}(\x)\pd{e^{pN}_i(\x,t)}{t}} =& 
        \inner{\omega_j(\x),\sum_{l=1}^{n}\ap_{li}\pd{e^{qN}_l(\x,t)}{\x}}+ \inner{\omega_j(\x),[B_p(\x)]_{i,-}\uv(t)}_{L^2}, 
\end{align*}
for every $j\in\{1,\dots, N\}$, leading us to
\begin{align}
    \L^{-1}_{pi} M^\top\dot{\ev}^p_i(t)=&-\sum_{l=1}^{n}\frac{\ap_{li}}{h}D^\top \ev^{p}_l(t) -\dfrac{\kappa_i}{h}\tr\tr^\top \ev^p_i(t) +B_{pi} \uv(t).  \label{eq:aux1}
\end{align}

Repeating this procedure for every co-energy variable, we obtain the following numerical model for system \eqref{eq:PHS_3}--\eqref{eq:BC_a2}.
\begin{align}
    \begin{bmatrix}
        \Lh_{q}^{-1} & 0\\ 0& \Lh_{p}^{-1}
    \end{bmatrix}\begin{bmatrix}
        \Mh & -\Bh_1\\
        0 & \Mh^\top
    \end{bmatrix}\begin{bmatrix}
        \dot{\ev}^q(t)\\\dot{\ev}^p(t)
    \end{bmatrix} =\dfrac{1}{h}\begin{bmatrix}
        0 & \Ah \Dh\\
        -\Ah^\top\Dh^\top & -\Bh_2
    \end{bmatrix}\begin{bmatrix}
        {\ev}^q(t)\\\ev^p(t)
    \end{bmatrix}+\begin{bmatrix}
        \bar{B}_q\\\bar{B}_p
    \end{bmatrix}\uv(t).\label{eq:PHSd}
\end{align}
where the state vector $\ev^q(t)\in\mathbb{R}^{nN}$ and $\ev^p(t)\in\mathbb{R}^{nN}$ are
$$\ev^q(t)=\begin{bmatrix}
    \ev^q_1(t)\\\vdots\\\ev^q_n(t)
\end{bmatrix}\quad \text{and}\quad \ev^p(t)=\begin{bmatrix}
    \ev^p_1(t)\\\vdots\\\ev^p_n(t)
\end{bmatrix},$$
respectively.

First, consider the following properties to verify whether discretization conserves the port-Hamiltonian structure. 

\begin{lemma}\label{lem:MatrixProp}
Let the matrices $\Ah$, $\Bh_1$,$\Bh_2$, $\Mh$ and $\Dh$ defined above. Then, the following identities hold. 

\begin{center}
\begin{tabular}{l l l l}
    {\rm(B1)} $\Mh^\top \Dh = \Dh\Mh^\top$, & {\rm (B2)} $\Mh^{-1}\Dh^\top=\Dh^\top \Mh^{-1}$, & {\rm (B3)} $\dfrac{1}{2}\Dh^\top +\Mh=I_{nN}$, & {\rm (B4)} $\Ah^{-\top}\Mh^{-1}=\Mh^{-1}\Ah^{-\top}$,\\
    {\rm (B5)} $\Ah\Dh=\Dh\Ah$, & {\rm (B6)} $\Ah^{\top}\Mh^{-1}=\Mh^{-1}\Ah^{\top}$, & \multicolumn{2}{l}{  {\rm (B7)} $\Ah^\top \Dh^\top \Mh^{-1}=\Mh^{-1} \Dh^\top \Ah^\top$. }
\end{tabular}    
\end{center}

\end{lemma}
\begin{proof}

It was shown in \cite{DelReyFernandez2023} that matrices $M$ and $D$ satisfy the following properties 
    $${\rm(I)}\; M^\top D = DM^\top, \quad {\rm(II)}\; M^{-1}D^\top=D^\top M^{-1}, \quad \text{and} \quad {\rm(III)}\; \dfrac{1}{2}D^\top +M=I_N.$$

As a consequence, (B1) is obtained using \cite[Proposition 9.1.6]{Bernstein2018}, that is, $\Mh^\top \Dh =(I_n\otimes M^\top )(I_n\otimes D )=I_n\otimes M^\top D$, then from (I) we get $I_n\otimes M^\top D=I_n\otimes  D M^\top= (I_n\otimes D )(I_n\otimes M^\top ) =\Dh\Mh^\top$. The proof of (B2) follows the same procedure and uses (II).
    For (B3) we use \cite[Proposition 9.1.4]{Bernstein2018} obtaining $\dfrac{1}{2}\Dh+\Mh=\dfrac{1}{2}I_n \otimes D+I_n\otimes M=I_n\otimes \left(\dfrac{1}{2}D +M\right)$, then, from (III), $I_n\otimes \left(\dfrac{1}{2}D +M\right)=I_n\otimes I_N=I_{nN}$.

    Identity (B4) is obtained using \cite[Proposition 9.1.6]{Bernstein2018}  as follows,  $\Ah^{-\top}\Mh^{-1}=(A^{-\top} \otimes  I_N)(I_n \otimes M^{-1})=A^{-\top}\otimes M^{-1} = (I_n \otimes M^{-1})(A^{-\top} \otimes  I_N)=\Mh^{-1} \Ah^{-\top}$.
    Identities (B5) and (B6) follow the same procedure as (B4).
    (B7) is obtaining applying (B2), (B5), and (B6).
\end{proof}

\begin{prop}\label{prop:PHS}
    Define the matrices 
    \begin{align}
    \begin{split}
          &Q_d=  h\begin{bmatrix}
        \Mh & -\Bh_1\\
        0 & \Mh^\top
    \end{bmatrix}, \quad 
    S_d=  \begin{bmatrix}
        \Lh_q^{-1} & 0\\ 0& \Lh_p^{-1}
    \end{bmatrix}\begin{bmatrix}
        \Mh & -\Bh_1\\
        0 & \Mh^\top
    \end{bmatrix}, \quad B_d=\begin{bmatrix}
        \Bh_q\\\Bh_p
    \end{bmatrix}, \\ & R_d=\begin{bmatrix}
        0 & 0 \\ 0 & \dfrac{\Mh^{-1}\Bh_2\Mh^{-\top}}{h^2}
    \end{bmatrix},
      \quad   J_d=\dfrac{1}{h^2}\begin{bmatrix}
         0 &   \Ah \Dh \Mh^{-\top}\\
      - \Mh^{-1} \Dh^\top \Ah^\top  & 0
    \end{bmatrix}.
    \end{split}
    \label{eq:QdSd}
\end{align}
 System \eqref{eq:PHSd} can be expressed as the port-Hamiltonian formulation
\begin{align}
    S_d\begin{bmatrix}
        \dot{\ev}^q(t)\\\dot{\ev}^p(t)
    \end{bmatrix}=(J_d-R_d)Q_d \begin{bmatrix}
        \ev^q(t) \\ \ev^p(t)
    \end{bmatrix}+B_d\uv(t),\label{eq:PHSd2}
\end{align}
 with the discrete energy 
\begin{align}
    \Ha_d(t)=\dfrac{1}{2} \inner{S_d\begin{bmatrix}
        \ev^q(t)\\ \ev^p(t)
    \end{bmatrix} , Q_d \begin{bmatrix}
        \ev^q(t) \\ \ev^p(t)
    \end{bmatrix}},
\end{align}
satisfying the rate of change 
\begin{align}
    \begin{split}
        \dot{\Ha}_d=&-\begin{bmatrix}
        \ev^q(t) \\ \ev^p(t)
    \end{bmatrix}^\top Q_d^\top   \begin{bmatrix}
        0 & 0 \\ 0 & \dfrac{\Mh^{-1}\Bh_2\Mh^{-\top}}{h^2}
    \end{bmatrix}   Q_d \begin{bmatrix}
        \ev^q(t) \\ \ev^p(t)
    \end{bmatrix}+ \begin{bmatrix}
        \ev^q(t) \\ \ev^p(t)
    \end{bmatrix}^\top Q_d^\top B_d \uv(t), \\ 
    =& -\begin{bmatrix}
        \ev^q(t) \\ \ev^p(t)
    \end{bmatrix}^\top S_d^\top   \begin{bmatrix}
        0 & 0 \\ 0 & \Lh_p\Mh^{-1}\Bh_2\Mh^{-\top} \Lh_p
    \end{bmatrix}   S_d \begin{bmatrix}
        \ev^q(t) \\ \ev^p(t)
    \end{bmatrix}+ \begin{bmatrix}
        \ev^q(t) \\ \ev^p(t)
    \end{bmatrix}^\top Q_d^\top B_d \uv(t).
    \end{split}
\end{align}

\end{prop}

\begin{proof}
 Using $S_d$, $B_d$ and $Q_d$ defined in \eqref{eq:QdSd},   from \eqref{eq:PHSd} we have
    \begin{align*}
    S_d\begin{bmatrix}
        \dot{\ev}^q(t) \\ \dot{\ev}^p(t)
    \end{bmatrix} 
    =& \frac{1}{h^2} \begin{bmatrix}
        0 & \Ah\Dh\\
        -\Ah^\top \Dh^\top& -\Bh_2
    \end{bmatrix}\begin{bmatrix}
        \Mh & -\Bh_1\\
        0 & \Mh^\top
    \end{bmatrix}^{-1} Q_d\begin{bmatrix}
        {\ev}^q(t) \\ \ev^p(t)
    \end{bmatrix} +  B_d \uv(t).
    \end{align*}

 Using properties (B7), (B6) and (B3) we obtain
 \begin{align*}
    \begin{bmatrix}
        0 & \Ah\Dh\\
        -\Ah^\top \Dh^\top& -\Bh_2
    \end{bmatrix}\begin{bmatrix}
        \Mh & -\Bh_1\\
        0 & \Mh^\top
    \end{bmatrix}^{-1}=&\begin{bmatrix}
      0 &   \Ah \Dh \Mh^{-\top}\\
      -\Ah^\top \Dh^\top \Mh^{-1}  & -\Ah^\top \Dh^\top \Mh^{-1}\Bh_1\Mh^{-\top} -\Bh_2\Mh^{-\top}
    \end{bmatrix}\\
    =& \begin{bmatrix}
      0 &   \Ah \Dh \Mh^{-\top}\\
      - \Mh^{-1} \Dh^\top \Ah^\top  & - \dfrac{1}{2}\Dh^\top \Mh^{-1}\Bh_2\Mh^{-\top} -\Bh_2\Mh^{-\top}
    \end{bmatrix}\\
    =& \begin{bmatrix}
      0 &   \Ah \Dh \Mh^{-\top}\\
      -  \Mh^{-1} \Dh^\top \Ah^\top  & - \left(\dfrac{1}{2}\Dh^\top+\Mh\right) \Mh^{-1}\Bh_2\Mh^{-\top} 
    \end{bmatrix}\\
    = & \begin{bmatrix}
      0 &   \Ah \Dh \Mh^{-\top}\\
      -  \Mh^{-1} \Dh^\top \Ah^\top  & - \Mh^{-1}\Bh_2\Mh^{-\top}  
    \end{bmatrix}.
\end{align*}

Then, 
\begin{align*}
    S_d\begin{bmatrix}
        \dot{\ev}^q(t) \\ \dot{\ev}^p(t)
    \end{bmatrix} 
    =& \left( \dfrac{1}{h^2}\begin{bmatrix}
         0 &   \Ah \Dh \Mh^{-\top}\\
      - \Mh^{-1} \Dh^\top \Ah^\top  & 0
    \end{bmatrix}- \begin{bmatrix}
        0 & 0 \\ 0 & \dfrac{\Mh^{-1}\Bh_2\Mh^{-\top}}{h^2}
    \end{bmatrix}\right)Q_d\begin{bmatrix}
        {\ev}^q(t) \\ \ev^p(t)
    \end{bmatrix} + B_d \uv(t),
\end{align*}
completing the proof.    
\end{proof}

\section{Uniform  exponential stability}\label{sec:Multiplier}

In this section, we will use the  multiplier method to obtain a bound on  the exponential decay rate of the approximated system obtained  in Section \ref{sec:MFEM}, assuming $\uv(t)=0$.
Let $C$ be the $N\times N$ matrix defined as
\begin{align}
    C=&\dfrac{1}{2}\begin{bmatrix}
        0 & 1 \\
         -2 &0 & 2\\
          & -3 &0 & 3 \\
            &&\ddots & \ddots & \ddots        \\
           &&& 1-N &0 & N-1\\
           &&& &-2N & 2N
    \end{bmatrix},\label{eq:C}
\end{align}
\begin{align}
\Ch & =I_n\otimes C \nonumber \\
     \Wh &=-h\Dh^{-\top}\Ch^\top\Mh\label{eq:Wh}
 \end{align}
satisfying $\Ah^{-\top}\Wh=\Wh\Ah^{-\top}$.

Define
\begin{align}
    V_{\epsilon}(t)=-\dfrac{h}{2}\inner{S_d\begin{bmatrix}
        \ev^q\\ \ev^p
    \end{bmatrix},\begin{bmatrix}
        0 & \Ah\\\Ah^\top & 0
    \end{bmatrix}^{-1}\begin{bmatrix}
        \Wh^\top & 0 \\ 0 & \Wh
    \end{bmatrix} S_d\begin{bmatrix}
        \ev^q\\ \ev^p
    \end{bmatrix}}, \label{eq:Ve}
    \end{align}
and consider the Lyapunov function candidate
\begin{align}
    V(t)=&\Ha_d(t)+\epsilon V_\epsilon(t).\label{eq:V}
\end{align}
  The Lyapunov function in \eqref{eq:V}--\eqref{eq:Ve} is a discretized version of the functional in  \cite{Mora2023} for boundary-damped port-Hamiltonian systems \eqref{eq:PHS_2}--\eqref{eq:BC_a}, with matrix $\Wh$ playing the role of the multiplier function $m(\x)$.
The classical multiplier approach  \cite{Tucsnak2009} will be used. It will be shown that  $V$ is bounded from above and below by $H_d$ and decays exponentially. This is then used to show the numerical approximation  \eqref{eq:PHSd2} is exponentially stable. More precisely, 
it will be shown  that $V_\epsilon$ satisfies the following two conditions: There are $\epsilon_0>0$, $\delta^d$ and $\epsilon_1$, independent of the mesh size, such tat
\begin{itemize}
    \item [(C1)] 
    \begin{align}
        |V_\epsilon (t) |\leq& \dfrac{1}{\epsilon_0} \Ha_d (t),\label{eq:VeCond1}
    \end{align}
    \item[(C2)]  
    \begin{align}
        \dot{V}_{\epsilon} (t) \leq& -\delta^d \Ha_d (t)  -\dfrac{1}{\epsilon_1}\dot{\Ha}_d (t).\label{eq:VeCond2}
    \end{align}
\end{itemize}

The following two lemmas are needed.
\begin{lemma}\label{lem:4}
    Let $\boldsymbol{g}$ be a vector defined as $\boldsymbol{g}=[\boldsymbol{g}_1, \dots, \boldsymbol{g}_n]^\top \in \mathbb{R}^{nN}$ with $\boldsymbol{g}_j\in\mathbb{R}^{N}$ for any $j\in\{1,\dots, n\}$, and $P_1=P_1^{\top}\in\mathbb{R}^{2n \times 2n}$ as defined in Section \ref{sec:PHS}. Matrices $\Ah$ and $\Wh$ satisfy the following identities:
    \begin{align}
        \max eig\left(\begin{bmatrix}
            0 & \Ah\\ \Ah^\top & 0
        \end{bmatrix}^{-2}\right)=&\max eig(P_1^{-2})
\text{ and }
         \|\Wh\boldsymbol{g}\|^2=\|\Wh^\top\boldsymbol{g}\|^2\leq \ell^2 \|\boldsymbol{g}\|^2,
    \end{align}
    where $\ell=\xr-\xl$ is the length of the spatial domain.
\end{lemma}

\begin{proof}
See the appendix section.
\end{proof}

\begin{lemma}\label{lem:inner}
    Let $\bp(\x)$ be a strictly positive continuous function on interval $[\xl,\xr]$, and $\bp_{k}=\bp(\x_k)$ be it value at node $\x_k$ and $\bpi_k$ the corresponding inverse. $\widetilde{\L}$ and $O$ be $N \times N$ matrices defined as $\widetilde{\L}=diag(\bp_1,\dots,\bp_N)$ and 
    \begin{align}
        {O}=& \dfrac{1}{2}\begin{bmatrix}
        0 & \bp_{2}-\bp_{1} & \\
        \bp_{2}-\bp_{1}& 0 & 2\left(\bp_{3}- \bp_{2}\right) \\
         &2\left(\bp_{3}- \bp_{2}\right) & 0 & \ddots 
         \\
          & &\ddots & \ddots & \ddots \\
          & & & \ddots & \ddots & (N-1)\left(\bp_{N}-\bp_{N-1}\right)\\
                 & & & & (N-1)\left(\bp_{N}-\bp_{N-1}\right) & 0 
    \end{bmatrix}.\label{eq:Opi}
    \end{align}    

    Given $M$, $D$ and $\tr$ as defined at \eqref{eq:MDtr} and $C$ at \eqref{eq:C}, the following statements hold for any $\boldsymbol{f}\in \mathbb{R}^{N}$:
    \begin{align}
        h\inner{ M \widetilde{\L}^{-1}  M^\top \boldsymbol{f},C \boldsymbol{f}}\leq&    -\dfrac{h}{2}\boldsymbol{f}^\top M \widetilde{\L}^{-1} \left(\widetilde{\L}-O\right) \widetilde{\L}^{-1} M^\top \boldsymbol{f} +\dfrac{\ell \bpi_{N}}{2 } \boldsymbol{f}^\top \tr \tr^\top \boldsymbol{f} \label{eq:iden1}\\
        -\dfrac{1}{h}\inner{D^\top \widetilde{\L} D \boldsymbol{f},C \boldsymbol{f}}=& -\dfrac{1}{2h}\boldsymbol{f}^\top D^\top \left( \widetilde{\L}-O\right) D\boldsymbol{f} -\dfrac{N\bp_{N}}{2h}\boldsymbol{f}^\top D^\top \tr\tr^\top D \boldsymbol{f} \label{eq:iden2}.
    \end{align}
\end{lemma}
\begin{proof}
See the appendix section.
\end{proof}

\begin{theorem}\label{thm:1}     
Consider $\xx\in\{p,q\}$ and $i\in\{1,\dots,n\}$. Let $\L_{\xx i}$ and $O_{\xx i}$ be matrices constructed as at \eqref{eq:Qxi} and \eqref{eq:Opi}, respectively, from the values of $\bp^{\xx}_{i}(\x)$ at discretization nodes.     If there exists a positive constant $\delta^d$, independent of $N$, such that for every  parameter $\bp_i^{\xx}$, the corresponding matrices $\L_{\xx i}$ and $O_{\xx i}$ satisfy the inequality
     \begin{align}
         \L_{\xx i}-O_{\xx i}\geq \delta^d \L_{\xx i}.\label{eq:IneqCond}
     \end{align}
     Then, $\Ha_d(t)$ decays exponentially at least with a uniform rate of $\alpha^d=\dfrac{\delta^d\epsilon \epsilon_0}{\epsilon+\epsilon_0}$ for any $\epsilon\in(0,\min\{\epsilon_0,\epsilon_1\}]$, where $\epsilon_0=\dfrac{\eta_\L}{\ell \mu_{P_1}}$ and $\epsilon_1=\dfrac{2\eta_K}{\ell \mu_\Psi}$.
\end{theorem}

\begin{proof}

    Condition (C1) can be verified using the Cauchy-Schwartz inequality and Lemma \ref{lem:4} as follows 
\begin{align*}
     |V_{\epsilon}(t)|=&\dfrac{h}{2}\left|\inner{S_d\begin{bmatrix}
        \ev^q(t) \\ \ev^p(t)
    \end{bmatrix},\begin{bmatrix}
        0 & \Ah \\ \Ah^\top & 0
    \end{bmatrix}^{-1}\begin{bmatrix}
        \Wh^\top & 0 \\ 0 & \Wh
    \end{bmatrix} S_d\begin{bmatrix}
        \ev^q(t) \\ \ev^p(t)
    \end{bmatrix}}\right| ,\\
    \leq& \dfrac{h}{2} \left\| S_d\begin{bmatrix}
        \ev^q(t)\\ \ev^p(t)
    \end{bmatrix} \right\| \left\| \begin{bmatrix}
        0 & \Ah\\ \Ah^\top & 0
    \end{bmatrix}^{-1}\begin{bmatrix}
        \Wh^\top & 0 \\ 0 & \Wh
    \end{bmatrix} S_d\begin{bmatrix}
        \ev^q(t) \\ \ev^p(t)
    \end{bmatrix}\right\|, \\
    \leq& \dfrac{h \ell \mu_{P_1} }{2} \left\| S_d\begin{bmatrix}
        \ev^q(t)\\ \ev^p(t)
    \end{bmatrix} \right\|^2 
    =\dfrac{\ell \mu_{P_1}}{\eta_{Q}}\left(\dfrac{h \eta_Q}{2} \left\| S_d\begin{bmatrix}
        \ev^q(t) \\ \ev^p(t)
    \end{bmatrix} \right\|^2\right).\nonumber
\end{align*}

From \eqref{eq:Qxi} we have $\disp\min \eig \left(\L_{\xx i}\right)=\min_{j\in[1,N]} \theta^{\xx}_{i,j}\geq \min_{\x\in[\xl,\xr]} \bp^\xx_i(\x)$. As a consequence, $\disp \min \eig \left(\begin{bmatrix}
    \Lh_q & 0 \\ 0 & \Lh_p
\end{bmatrix} \right)\geq \min_{\x\in[\xl,\xr] } eig( \L(\x) )=\eta_Q$. Then, $\Ha_d(t)=\dfrac{h}{2}\inner{ S_d, \begin{bmatrix}
        \ev^q(t) \\ \ev^p(t)
    \end{bmatrix}, \begin{bmatrix}
        \Lh_q & 0 \\ 0 & \Lh_p
    \end{bmatrix} S_d \begin{bmatrix}
        \ev^q(t) \\ \ev^p(t)
    \end{bmatrix} } \geq \dfrac{h}{2} \eta_{Q} \left\| S_d\begin{bmatrix}
        \ev^q(t) \\ \ev^p(t)
    \end{bmatrix} \right\|^2$.
This implies that $|V_\epsilon(t)|\leq \dfrac{\ell \mu_{P_1}}{\eta_Q}\Ha_d(t)$, that is, \eqref{eq:VeCond1} holds with
$\epsilon_0=\dfrac{\eta_\L}{\ell \mu_{P_1}}.$

On the other hand, analyzing the time derivative of $V_\epsilon(t)$ using \eqref{eq:QdSd}, \eqref{eq:Wh} and Lemma \ref{lem:MatrixProp} identities, we obtain
\begin{align*}
    \dot{V}_{\epsilon}(t)=&- h\inner{S_d\begin{bmatrix}
        \ev^q(t)\\ \ev^p(t)
    \end{bmatrix},\begin{bmatrix}
        \Wh & 0 \\ 0 & \Wh^\top
    \end{bmatrix} \begin{bmatrix}
        0 & \Ah\\ \Ah^\top & 0
    \end{bmatrix}^{-1} S_d\begin{bmatrix}
        \dot{\ev}^q(t)\\ \dot{\ev}^p(t)
    \end{bmatrix}},\\
    =&- \frac{1}{h}\inner{S_d\begin{bmatrix}
        \ev^q(t)\\ \ev^p(t)
    \end{bmatrix}, \begin{bmatrix}
        0 & \Wh \Ah^{-\top}\\\Wh^\top \Ah^{-1} & 0
    \end{bmatrix}  \begin{bmatrix}
      0 &   \Ah \Dh \Mh^{-\top}\\
      -  \Mh^{-1} \Dh^\top \Ah^\top  & - \Mh^{-1}\Bh_2\Mh^{-\top} 
    \end{bmatrix} Q_d\begin{bmatrix}
        {\ev}^q(t)\\\ev^p(t)
    \end{bmatrix}},\\
    =&- \inner{S_d\begin{bmatrix}
        \ev^q(t)\\ \ev^p(t)
    \end{bmatrix}, \begin{bmatrix}
       h\Dh^{-\top}\Ch^\top  \Dh^\top \Lh_q & 2h\Dh^{-\top}\Ch^\top \Bh_1\Mh^{-\top} \Lh_p\\ 0 & -h\Mh^\top\Ch  \Mh^{-\top} \Lh_p
    \end{bmatrix}  S_d\begin{bmatrix}
        {\ev}^q(t)\\\ev^p(t)
    \end{bmatrix}}.
\end{align*}

Define $\ub(t)=[\ub_1^\T(t),\ldots,\ub_n^\T(t)]^\top$ and $\vb(t)=[\vb_1^\T(t),\ldots,\vb_n^\T(t)]^\top$ with $\ub_i(t),\vb_i(t)\in\mathbb{R}^{N}$ for all $i\in\{1,\dots,n\}$.
Considering the transformation
\begin{align}
    \begin{bmatrix}
        \dfrac{\Dh}{h} & 0 \\0 & \Lh_p^{-1} \Mh^\top
    \end{bmatrix}\begin{bmatrix}
        \ub (t) \\ \vb(t)
    \end{bmatrix} = & S_d\begin{bmatrix}
        {\ev}^q(t)\\\ev^p(t)
    \end{bmatrix},\label{eq:Tranformation}
\end{align}
the discrete energy is rewritten as $\Ha_d(t)=\dfrac{1}{2h}\ub^\T (t) \Dh^\top\Lh_q \Dh \ub(t)+\dfrac{h}{2}\vb^\top(t) \Mh\Lh_{p}^{-1}\Mh^\top \vb(t)$, $\dot{\Ha}_d(t)=-\vb^{\top}(t) (I_n\otimes\tr) K (I_n\otimes\tr^{\top}) \vb(t)$, and $\dot{V}_\epsilon (t)$ is expressed as
\begin{align}
    \dot{V}_\epsilon (t)=&h\inner{\Mh \Lh_{p}^{-1} \Mh^\top \vb(t),\Ch\vb(t)} -\dfrac{1}{h}\inner{  \Dh^\top \Lh_q \Dh \ub(t), \Ch\ub(t)} - 2\inner{\Ch \ub(t),\Bh_1\vb(t)} \label{eq:dVe}
\end{align}

Since $(I_n\otimes \tr^\top)\Ch\ub(t)= N (I_n\otimes \tr^\top)\Dh\ub(t)$ and $\L_{\xx}(\xr)=diag(\bp^{\xx}_{1,N},\dots \bp^{\xx}_{n,N})$, if \eqref{eq:IneqCond} holds, then from  Lemma \ref{lem:inner} the inner products $\inner{\Mh \Lh_{p}^{-1} \Mh^\top \vb(t),\Ch\vb(t)}=\disp \sum_{i=1}^{n}h\inner{\M \L_{pi}^{-1} \M^\top \vb_i(t),\C\vb_i(t)}$, $\inner{\Dh^\top \Lh_q \Dh \ub(t), \Ch\ub(t)}=\disp\sum_{i=1}^{n}\inner{\D^\top \L_{qi} \D \ub_i(t), \C^\top\ub_i(t)}$, and $2\inner{\Ch\ub(t),\Bh_1\vb(t)}=\inner{\Ch\ub(t),(I_n\otimes \tr) A^{-T}K(I_n\otimes \tr^\top)\vb(t)}$ lead us
 \begin{align}
    \sum_{i=1}^{n}h\inner{\M \L_{pi}^{-1} \M^\top \vb_i(t),\C\vb_i(t)}
    =& \sum_{i=1}^{n}
      -\dfrac{h}{2}\vb_i^\top(t) M \L_{pi}^{-1}\left( \L_{pi}-O_{pi}\right) \L_{p}^{-1}M^\top \vb_i(t)     +\dfrac{\ell\bpi_{i,N}^p}{2 } \vb_i^\top (t) \tr \tr^\top \vb_i(t),\nonumber\\
    \leq  & \sum_{i=1}^{n}
      -\dfrac{h\delta^d}{2}\vb_i^\top (t) M \L_{pi}^{-1}M^\top \vb_i (t)     +\dfrac{\ell\bpi_{i,N}^p}{2 } \vb_i^\top(t) \tr \tr^\top \vb_i(t),\nonumber\\
    \leq&   -\dfrac{h \delta^d}{2}\vb^\top(t) \Mh \Lh_{p}^{-1}\Mh^\top \vb(t)    +\dfrac{\ell }{2 } \vb^\top (t)(I_n\otimes \tr) \L^{-1}_p(\xr) (I_n\otimes\tr^\top) \vb(t), \label{eq:inn1}\\
     -\sum_{i=1}^{n}\dfrac{1}{h}\inner{\D^\top \L_{qi} \D \ub_i(t), \C^\top\ub_i(t)}
    =&
     -\dfrac{1}{2h}\sum_{i=1}^{n}\ub_i^\top (t) D^\top \left(\L_{qi}-O_{qi}\right) D\ub_i (t)+N\bp_{i,N}^q \ub_i^\top (t)D^\top \tr\tr^\top D \ub_i(t)\nonumber\\
     \leq&
     -\dfrac{1}{2h}\sum_{i=1}^{n}\delta^d\ub_i^\top(t) D^\top \L_{qi} D\ub_i(t) +N\bp_{i,N}^q \ub_i^\top(t) D^\top \tr\tr^\top D \ub_i(t)\nonumber\\
     \leq&
     -\dfrac{\delta^d\ub^\top (t)\Dh^\top \Lh_{q} \Dh\ub(t) -N \ub^\top (t)\Dh^\top (I_n\otimes\tr) \Lh_{q}(\xr)(I_n\otimes\tr^\top) \Dh \ub(t)}{2h}
    ,\label{eq:inn2}
\end{align}
and
\begin{align}
    -\inner{\Ch\ub(t),(I_n\otimes \tr) A^{-T}K(I_n\otimes \tr^\top)\vb(t)}= & -N\inner{(I_n\otimes \tr^\top)\Dh\ub(t), A^{-T}K(I_n\otimes \tr^\top)\vb(t)} \nonumber\\
    \leq& \dfrac{N}{2h} \ub^\T(t) \Dh^\top (I_n\otimes \tr) \L_q(\xr) (I_n\otimes \tr^\top)\Dh \ub(t) \nonumber\\
    &+ \dfrac{\ell}{2} \vb^\T (t) (I_n\otimes \tr) K A^{-1} \L^{-1}_q(\xr) A^{-\top}K(I_n\otimes \tr^\top)\vb(t),\label{eq:inn3}
\end{align}
respectively.
Substituting \eqref{eq:inn1}--\eqref{eq:inn3} into \eqref{eq:dVe}, the time derivative of the multiplier function is bounded as
\begin{align}
    \dot{V}_\epsilon (t)\leq& -\dfrac{h\delta^d}{2}\vb^\top (t) \Mh\Lh_p^{-1}\Mh^\top \vb(t)  -\dfrac{\delta^d}{2h}\ub^\top(t)  \Dh^\top \Lh_q \Dh  \ub(t)\nonumber\\
    & +\dfrac{\ell}{2 } \vb^\top(t) \left( I_n\otimes\tr\right) \L^{-1}_p(\xr) \left( I_n\otimes\tr^\top\right) \vb(t) + \dfrac{\ell}{2} \vb(t) (I_n\otimes \tr) K A^{-1} \L^{-1}_q(\xr) A^{-\top}K(I_n\otimes \tr^\top)\vb(t)\nonumber\\
    \leq& -\delta^d \Ha_d(t)+ \dfrac{\ell}{2 } \vb^\top(t) \left( I_n\otimes\tr\right) \Psi \left( I_n\otimes\tr^\top\right) \vb (t)\nonumber\\
    \leq&  -\delta^d \Ha_d(t)-\dfrac{\ell \mu_\Psi}{2\eta_K}\dot{\Ha}_d(t)
\end{align}




This implies that if the inequality \eqref{eq:IneqCond} holds, then, condition (C2) is satisfied with $\epsilon_1=\dfrac{2\eta_K}{\ell \mu_\Psi}$, guarantying that the energy $H_d$ of model \eqref{eq:PHS_3} decay  exponentially with a rate of at least $\alpha^d=\dfrac{\delta^d \epsilon\epsilon_0}{\epsilon+\epsilon_0}$, $\forall \epsilon\in(0,\min\{\epsilon_0,\epsilon_1\}]$. 
\end{proof}

    Note that the bound on numerical approximation's exponential decay rate, $\alpha^d$, equals the exponential decay rate for the continuous system if $\delta^d=\delta^c$.

\subsection{Physical parameters: required properties.}


The question remains of determining when  inequality \eqref{eq:IneqCond} holds.
Note that \eqref{eq:IneqCond}, is equivalent to  that $(1-\delta^d)\L_{\xx i}-O_{\xx i}$ is a semi-definite positive matrix. It is the discrete analogue of Assumption 1. 

If for each pair of matrices $(\L_{\xx i},O_{\xx i})$ associated with parameter $\bp_i^{\xx}(\x)$, there is a $c>0$, independent of $N$, such that 
\begin{align}
    (1-c)\L_{\xx i}-O_{\xx i}>0,\label{eq:cond1}
\end{align}
then, $\delta^d$ is defined as the maximum $c$ such that $\eqref{eq:cond1}$ holds for every pair $(\L_{\xx i},O_{\xx i})$.
 From \eqref{eq:Qxi} and \eqref{eq:Opi} we have that $(1-c)\L_{\xx i}-O_{\xx i}$ is a tridiagonal matrix that depends on $c$ and the physical parameter $\bp_i^\xx(\x)$ evaluated at the  nodes. The positiveness of tridiagonal matrices has been widely studied in  the literature \cite{Andelic2011, Bernstein2018, Chien1998, Ismail1991, Ismail1992, Johnson1996}. This

 \begin{theorem}\label{thm:LargeN}
      For a sufficiently large number  $N$ of elements in the discretization, Assumption \ref{assump:1} implies that \eqref{eq:cond1} holds for every parameter $\theta_i^p, \theta_i^q$, $i=1..n.$
 \end{theorem}
 \begin{proof}
    Let $h=\dfrac{\ell}{N}$ be the mesh size, $\x_i=ih+\xl$ be the $i$-th node of the spatial discretization, and $\bp(\x)$ be a physical parameter with  $\bp_i=\bp(\x_i)$.
     From Assumption \ref{assump:1}, we have that for all $\x\in[\xl,\xr]$  the inequality   
     $$(1-\delta^c)\bp(\x)>(\x-\xl)\bp'(\x)$$
     holds for every $\bp(\x)\in\Theta$. Considering  the approximation $\bp'(\x_i)\approx \dfrac{\bp_{i+1}-\bp_i}{h}$   and $\x_i-\xl=ih$. Then, for an $N$ large enough the inequality $$(1-c)\bp_i>i(\bp_{i+1}-\bp_i), \quad\forall i\{1,\dots,N\}$$
     holds for some $c\in(0,\delta^c]$. This implies that matrix $(1-c)\L_{\xx i}-O_{\xx i}$ is diagonal dominant with strictly positive terms in the main diagonal. Consequently, $(1-c)\L_{\xx i}-O_{\xx i}$ is definite positive.
 \end{proof}

In some cases the required inequality holds regardless of the mesh size.
\begin{theorem}\label{thm:2}
    Let $\bp(\x)$ be a positive smooth function on interval $[\xl,\xr]$, with $\bp_i$ denoting the values of $\bp(\x)$ at node $\x_i$, $\bp_i=\bp(\x_i)$. $O$ and $\widetilde{Q}$ be the corresponding matrices as defined in Lemma \ref{lem:inner}. If $\bp$ satisfy one of the following conditions
    \begin{itemize}
        \item[\rm (D1)] $\dfrac{i^2\left(\bp_{i+1} -\bp_{i}\right)^2}{(1-c)^2\bp_{i+1}\bp_{i}} < \dfrac{1}{\cos^2\left(\dfrac{\pi}{N+1}\right)}$ for every $i\in\{1,\dots,N-1\}$,

        \item[\rm (D2)] $\disp \frac{8}{N(N-1)} \min_{i=1,3,\dots}\{ (1-c)\bp_{i}\} \min_{i=2,4,\dots}\{(1-c)\bp_{i}\}-\sum_{i <k} (1-c)^2 \left(\bp{i}-\bp_{k}\right)^2 <\sum_{i=1}^{N-1}i^2\left(\bp^\xx_{i+1}-\bp^\xx_{i}\right)^2$,

    \end{itemize}
    for some $0<c <1$, then, matrices $\widetilde{\L}$ and $O$ satisfy $(1-c)\widetilde{\L}-O >0$.
\end{theorem}
\begin{proof}
    Note that $(1-c)\widetilde{\L}-O$ is a symmetric tridiagonal matrix, 
    \begin{align*}
    (1-c)\widetilde{\L}-O=\begin{bmatrix}
        (1-c)\bp_{1} & -\frac{\bp_{2}-\bp_{1}}{2} & \\
        -\frac{\bp_{2}-\bp_{1}}{2} & \ddots  & & \\
            && (1-c)\bp_{i} & -\frac{j(\bp_{i+1}-\bp_{i})}{2} & \\
        &&-\frac{i(\bp_{i+1}-\bp_{i})}{2} & (1-c)\bp_{i+1}
        \\ &&&\ddots &\\
         &&& & -\frac{(N-1)(\bp_{N}-\bp_{N-1})}{2} & \\
        &&&-\frac{(N-1)(\bp_{N}-\bp_{N-1})}{2} & (1-c)\bp_{N}
    \end{bmatrix}
\end{align*}
with strictly positive elements in the main diagonal since $0<c<1$. According to \cite[Proposition 2.2]{Andelic2011} and \cite[Theorem 2.5 \& Proposition 2.5]{Johnson1996}, the inequality expressed in  (D1) is a sufficient condition to guarantee the positive definiteness of this matrix. Using the approach described in \cite[Theorem 2.3]{Chien1998}, we obtain the alternative sufficient condition shown in (D2).
As a consequence, if (D1) or (D2) holds, then $(1-c)\widetilde{\L}-O$ is definite positive.
\end{proof}
\begin{remark}\label{rmk:poss}
1. If a parameter is constant, then condition (D1) holds trivially. 

2.
     If condition {\rm(D1)} holds for some positive function $\bp(\x)$, then any function $\widetilde{\bp}(\x)$ of the form $\widetilde{\bp}(\x)=\gamma \bp(\x)$ or $\widetilde{\bp}(\x)=\dfrac{\gamma}{ \bp(\x)}$ where  $\gamma\in\mathbb{R}^+$ is a constant also satisfies this condition.
\end{remark}

If every physical parameter $\bp^\xx_i(\x)$ satisfies Theorem \ref{thm:2}, then condition \eqref{eq:IneqCond} in Theorem \ref{thm:1} holds, implying uniform conservation of the exponential decay rate.

\subsection{Example: Piezoelectric beam}\label{example:Piezo}

\begin{figure}
	\centering
	\includegraphics[width=0.75\columnwidth]{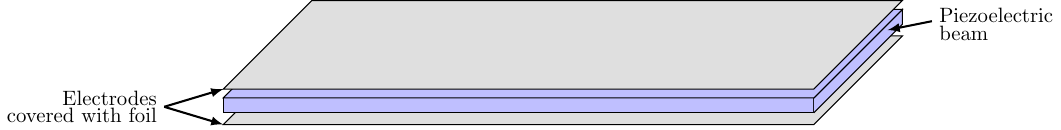}
	\caption{Piezoelectric beam}
	\label{fig:piezo}
\end{figure}

Consider the piezoelectric beam shown in Figure \ref{fig:piezo}.
Let $\rho_0$, $\alpha_0$, $\gamma$, $\mu_0$ and $\tau_0$ be the basic material density, elastic stiffness, piezoelectric coefficient, magnetic permeability, and dielectric permittivity of the beam and $\bp(\x)=\dfrac{10-\x}{10}$ be fundamental function that describes the spatial variations of the physical parameter on interval [0,1] due to changes in cross-section area, material compositions, etc. Then, the beam parameter are $\mu(\x)=\bp(\x)\mu_0$, $\rho(\x)=\bp(\x)\rho_0$, $\alpha(\x)=\bp(\x)\alpha_0$ and $\beta(\x)=\dfrac{\bp(\x)}{\tau_0}$, and the dynamics of the beam longitudinal displacement $v(\x,t)$ and electric charge $\phi(\x,t)$ are described through the following distributed-parameter system \cite{Morris2014}:
\begin{align}
	\begin{split}
		\pd{}{t}\left( \rho(\x)\pd{v(\x,t)}{t}\right)=&\pd{}{\x}\left((\alpha(\x) -\gamma^2 \beta(\x))\pd{v(\x,t)}{\x}\right)-\gamma\pd{}{\x}\left(\beta(\x)\pd{\phi(\x,t)}{\x}\right)\\
		\pd{}{t}\left( \mu(\x)\pd{\phi(\x,t)}{t}\right)=&\pd{}{\x}\left( \beta(\x)\pd{\phi(\x,t)}{\x}\right)-\gamma\pd{}{\x}\left(\beta(\x)\pd{v(\x,t)}{\x}\right)    
	\end{split}, \quad \forall \x\in[0,1]
\end{align}
subject to
\begin{align}
	v(0,t)=\phi(0,t)=0\\
	(\alpha(1)-\gamma^2\beta(1))\pd{v(1,t)}{\x}-\gamma\beta(1)\pd{\phi(1,t)}{\x}=&F(t)\\
	\beta(1)\pd{\phi(1,t)}{\x}-\gamma\beta(1)\pd{v(1,t)}{\x}=&V(t)
\end{align}
where $F(t)$ and $V(t)$ are the mechanical force and voltage applied at $\x=1$. Choosing $p_1(\x,t)=\rho(\x)\pd{v(\x,t)}{t}$, $p_2(\x,t)=\mu(\x)\pd{\phi(\x,t)}{t}$, $q_1(\x,t)=\pd{v(\x,t)}{\x}$ and $\displaystyle q_2(\x,t)=\pd{\phi(\x,t)}{\x}-\gamma\pd{v(\x,t)}{\x}$ we obtain the following port-Hamiltonian formulation:
\begin{align}
	\pd{}{t}\begin{bmatrix}
		q_1(\x,t)\\ q_2(\x,t)\\ p_1(\x,t) \\p_2(\x,t)
	\end{bmatrix}=&\begin{bmatrix}
		0 & 0 & 1 & 0\\
		0 & 0 & -\gamma & 1\\
		1 & -\gamma & 0 & 0\\
		0 & 1 & 0 & 0
	\end{bmatrix}\begin{bmatrix}
		\alpha(\x)q_1(\x,t)\\ \beta(\x)q_2(\x,t)\\ \dfrac{p_1(\x,t)}{\rho(\x)} \\ \dfrac{p_2(\x,t)}{\mu(\x)}
	\end{bmatrix}, \quad \forall \x\in[0,1]\label{eq:Pbeam1}\\
	\begin{split}
		p_1(0,t)=& p_2(0,t)=0\\
		\alpha(1)q_1(1,t)-&\gamma\beta(1)q_2(1,t)=F(t)\\
		\beta(1)q_2(1,t)=&V(t)
	\end{split}\label{eq:Pbeam2}
\end{align}

Considering the feedback 
$$\begin{bmatrix}
	F(t)\\V(t)
\end{bmatrix}=-\begin{bmatrix}
	k_1 & 0 \\. & k_2
\end{bmatrix}\begin{bmatrix}
	\dfrac{p_1(1,t)}{\rho(1)} \\ \dfrac{p_2(1,t)}{\mu(1)}
\end{bmatrix}$$
the piezoelectric beam behavior is described by \eqref{eq:PHS_2}--\eqref{eq:BC_a}, with
$A=\begin{bmatrix}
	1 & 0 \\-\gamma & 1
\end{bmatrix}$, $\bp^q_1(\x)=\bp(\x)\alpha_0$, $\bp^q_2(\x)=\dfrac{\bp(\x)}{\tau_0}$, $\bp^p_1(\x)=\dfrac{1}{\bp(\x)\rho_0}$ and $\bp^p_2(\x)=\dfrac{1}{\bp(\x)\mu_0}$. Note that all these physical parameters are proportional to the cross-section area $\bp(\x)$. Consequently, from \eqref{eq:1A}, is required that for all $\x$ on interval $[0,1]$, inequalities $\delta^c \leq \dfrac{\bp(\x)-m(\x) \bp'(\x)}{\bp(\x)} $ and $\delta^c\leq \dfrac{\bp(\x)+m(\x) \bp'(\x)}{\bp(\x)}$ hold simultaneously. Thus, Assumption \ref{assump:1} is satisfied defining 
\begin{align*}
	\disp\delta^c=&\inf_{\x\in[0,1]} \left\{\dfrac{\bp(\x)-m(\x) \bp'(\x)}{\bp(\x)}, \dfrac{\bp(\x)+m(\x) \bp'(\x)}{\bp(\x)} \right\}\\
	=&\inf_{\x\in[0,1]} \left\{\dfrac{10}{10-x}, \dfrac{10-2x}{10-x} \right\}\\
	=&\dfrac{8}{9}.
\end{align*}

On the other hand, given a uniform partition of the spatial domain, $h=\dfrac{1}{N}$, we get that
\begin{align}
	\bp_{i}=&\dfrac{10N-i}{10N}, \forall i\in\{0,\dots,N\}
\end{align}

Evaluating condition {(D1)} for $\bp(\x)$ we obtain that for every $i\in\{1,\dots,N-1\}$
\begin{align*}
	(1-c)^2 >& \cos^2\left(\dfrac{\pi}{N+1}\right) \dfrac{i^2(\bp_{i+1}-\bp_{i})^2}{ \bp_{i+1} \bp_{i}} \\
	>& \cos^2\left(\dfrac{\pi}{N+1}\right) \dfrac{i^2}{(10N-i-1) (10N-i)}
\end{align*}

Note that $\disp \dfrac{i^2}{(10N-i-1) (10N-i)} \leq \dfrac{(N-1)^2}{9 N (9N+1)}< \dfrac{1}{81}$ for any $i\in\{1,\dots,N\}$ and $\cos^2\left(\dfrac{\pi}{N+1}\right) < 1$  $\forall N$. Then, considering $(1-c)^2\geq \dfrac{1}{81}$,
condition {(D1)} holds for  $0< c \leq \dfrac{8}{9}$. Consequently, by Remark \ref{rmk:poss}, all physical parameters satisfy Theorem \ref{thm:2}. Thus, condition \eqref{eq:IneqCond} is satisfied with
$\delta^d=\sup c=\dfrac{8}{9}=\delta^c$, that is,, the mixed finite elements numerical method in Section \ref{sec:MFEM} lead us to a numerical approximation of \eqref{eq:Pbeam1}--\eqref{eq:Pbeam2} that have a strictly uniform conservation of the exponential decay rate, from the multiplier approach point of view.

\section{Application to linear quadratic control}


In this section we analyze the application of the mixed finite-element proposed for linear quadratic control, comparing the design results with those obtained from the standard finite-differences method.
 Consider the following anisotropic wave equation expressed as the co-energy formulation \eqref{eq:PHS_3}--\eqref{eq:BC_a2},
 \begin{align}
 	\begin{split}
 		  \begin{bmatrix}
 			\bp^q(\x) & 0\\0 & \bp^p(\x)
 		\end{bmatrix}^{-1}\pd{}{t}\begin{bmatrix}
 			e^q(\x,t)\\e^p(\x,t)
 		\end{bmatrix}-&\begin{bmatrix}
 			0 & 1\\1 & 0
 		\end{bmatrix}\pd{}{\x}\left(  \begin{bmatrix}
 			e^q(\x,t)\\e^p(\x,t)
 		\end{bmatrix} \right)=\begin{bmatrix}
 			0\\b(\x)
 		\end{bmatrix}u(t), \quad \forall \x\in[0,1] \\
 		e^p(0,t)=&0\\
 		e^q(1,t)=&-\kappa_1 e^p(1,t)
 	\end{split}\label{eq:waveeq}
 \end{align}
 with physical parameters $\bp^q(\x)=\bp(\x)\tau_0$ and $\bp^p(\x)=\dfrac{1}{\bp(\x)\rho_0}$ and 
 \begin{align}
 	b(\x)=\begin{cases}
 		3\times 10^4 x^2(x-0.1)^2, & \forall \x\in[0,0.1],\\
 		0, & \text{otherwise}
 	\end{cases},
 \end{align} 
and total energy
\begin{align}
	\Ha(t)=\dfrac{1}{2}\int_{0}^1 \begin{bmatrix}
			e^q(\x,t)\\e^p(\x,t)
		\end{bmatrix}^\top\begin{bmatrix}
			\bp^q(\x) & 0\\0 & \bp^p(\x)
		\end{bmatrix}^{-1} \begin{bmatrix}
			e^q(\x,t)\\e^p(\x,t)
		\end{bmatrix} \d\x
\end{align}
 
 For simplicity, we choose $\bp(\x)=\dfrac{10-\x}{10}, \forall \x\in[0,1]$, as in the previous example. This implies that $\bp^q(\x)$ and $\bp^p(\x)$ satisfy the conditions (D1) of Theorem \ref{thm:2}, as shown in Section \ref{example:Piezo}, and the approximated model obtained from the mixed finite-element scheme described in Section \ref{sec:MFEM} conserves uniformly the exponential decay rate.
 
 Now consider the infinite-time linear quadratic (LQ) controller design for $u(t)$ with cost functional
 \begin{align}
 	J(u(t),e^p(\x,0),e^q(\x,0))
=&\int_{0}^{\infty}20\Ha(t)+10^{-3} \inner{u(t),u(t)} \d t\nonumber\\
=&\int_{0}^{\infty}10\inner{\begin{bmatrix}
		e^q(\x,t)\\e^p(\x,t)
	\end{bmatrix},\begin{bmatrix}
		\bp^q(\x) & 0\\0 & \bp^p(\x)
	\end{bmatrix}^{-1} \begin{bmatrix}
		e^q(\x,t)\\e^p(\x,t)
\end{bmatrix}}_{L^2}+10^{-3} \inner{u(t),u(t)} \d t.
 \end{align}
whose solution is $u(t)=-K\begin{bmatrix}
	e^q(\x,t)\\e^p(\x,t)
\end{bmatrix}$ with $\disp K(\cdot)=10^3 \int_{0}^{1}\begin{bmatrix}
0 & b(x)
\end{bmatrix}\Pi (\cdot) \d\x$ and $\Pi$ is the unique positive semi-definite solution to
a Riccati operator equation. Unfortunately, the solution to this Riccati operator equation cannot be explicitly calculated \cite[Section 4.2]{Mora2023} and in the practice the control is calculated using numerical approximation models.Then, based on the model described in Proposition \ref{prop:PHS}, the cost function is rewritten as 
\begin{align}
	J_N(u(t),\ev^p(0),\ev^q(0))
	=&\int_{0}^{\infty}20\Ha_d(t)+10^{-3} \inner{u(t),u(t)} \d t\nonumber\\
	=&\int_{0}^{\infty}10\inner{S_d\begin{bmatrix}
			\ev^q(t)\\\ev^p(t)
		\end{bmatrix},Q_d \begin{bmatrix}
			\ev^q(t)\\\ev^p(t)
	\end{bmatrix}}+10^{-3} \inner{u(t),u(t)} \d t.
\end{align}
subject to \eqref{eq:PHSd2}, with $B_d=\begin{bmatrix}
	0\\\bar{B}_p
\end{bmatrix}$ and $\bar{B}_p\in\mathbb{R}^N$ defined as $[\bar{B}_p]_i=\disp\int_{0}^1\omega_i b(\x)\d\x$. The solution
\begin{align}
	u(t)=-K_d  \begin{bmatrix}
		\ev^q(t)\\\ev^p(t)
	\end{bmatrix},\label{eq:LQcontroller}
\end{align}
  where $K_d=10^3 [0 \quad \bar{B}_p^\top]\Pi_d$ and $\Pi_d$ is the solution to the algebraic Riccati equation
\begin{align}
	(J_d-R_d)^\top \Pi_d S_d+ S_d^\top \Pi_d (J_d-R_d)- S_d^\top \Pi_d \begin{bmatrix}
		0 \\ \bar{B}_p
	\end{bmatrix}^\top \begin{bmatrix}
	0 \\ \bar{B}_p
\end{bmatrix}\Pi_d  S_d +S_d^\top Q_d=0,
\end{align}
is obtained using the function $\mathsf{icare}$ of MATLAB R2021.

\begin{figure}
	\begin{tabular}{cc}
		FE model & MFEM model\\
		\includegraphics[width=0.475\columnwidth]{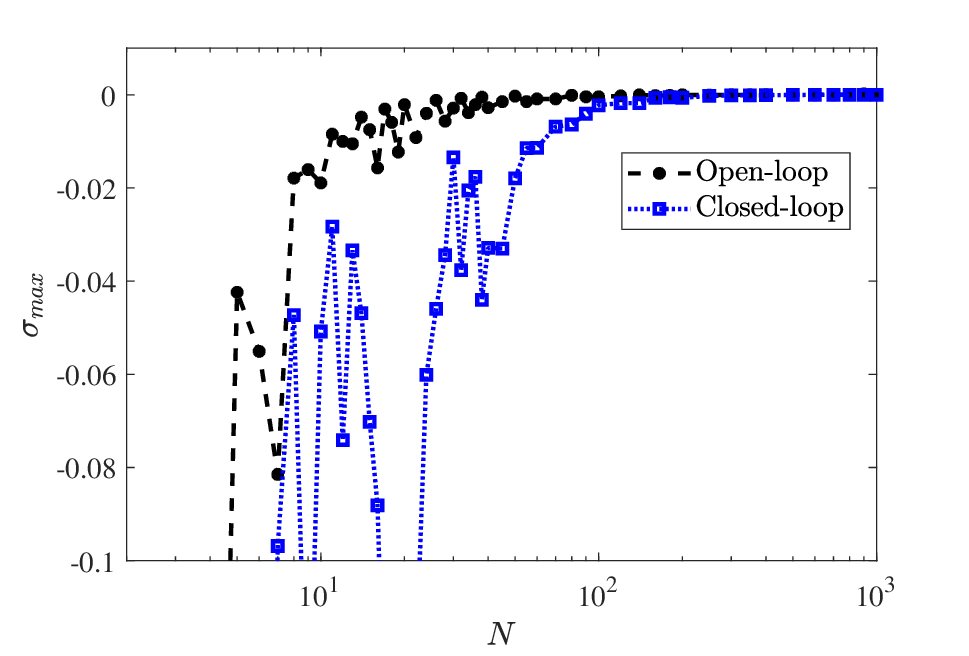}&
		\includegraphics[width=0.475\columnwidth]{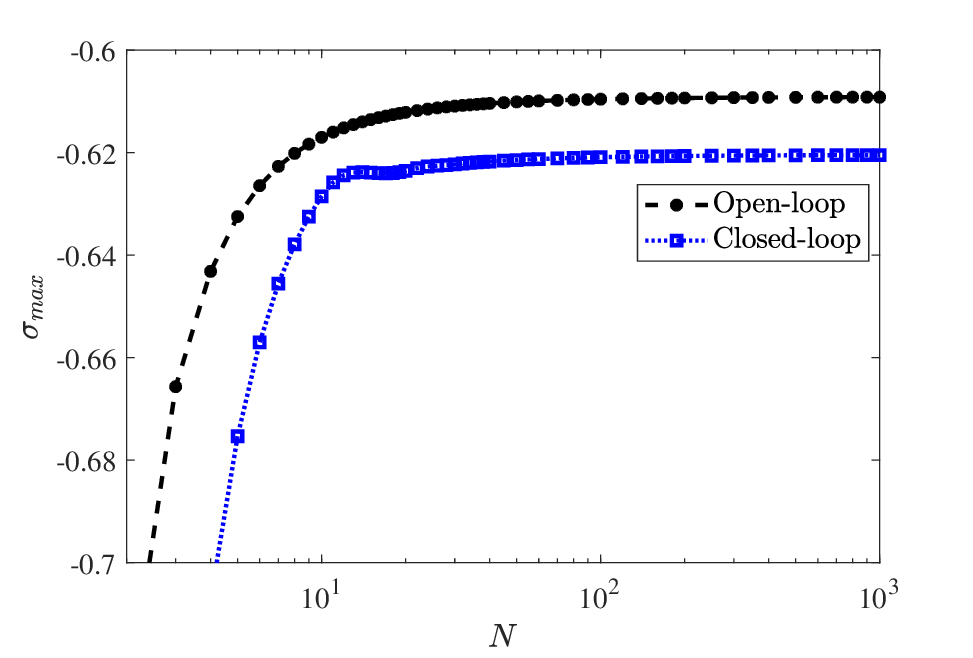}\\
		(a) & (b)
	\end{tabular}
	\caption{Maximum real part of the eigenvalues, $\sigma_{max}$, for the approximated model in open- and closed-loop, with $\rho_0=\tau_0=1$ and $\kappa_1=0.05$. (a): using the standard finite-element methods. (b): using the MFEM scheme proposed.}
	\label{fig:LQeigenvals}
\end{figure}

 Figure \ref{fig:LQeigenvals} shows a comparison of the maximum real part of the eigenvalues, $\sigma_{max}$, between the standard FE and the proposed MFEM approximation models of \eqref{eq:waveeq}, before and after closing the loop with the LQ controller \eqref{eq:LQcontroller}. The results are obtained assuming $\rho_0=\tau_0=1$ and $\kappa_1=0.5$. In Figure \ref{fig:LQeigenvals}.a we can observe how, for the FE model, $\sigma_{max}$ converges to 0, independently if the system is in open- or closed-loop, i.e., the system eigenvalues tends to the imaginary axis when $N$ increase, loosing the exponential stability, even with the LQ controller. For the proposed MFEM, there is a negative lower-upper bound for $\sigma_{max}$ for open-loop system, conserving the exponential stability uniformly, as proven analytically in Section \ref{sec:Multiplier}. The closed-loop system with the LQ controller presents the same behavior, slightly improving the exponential decay rate, as shown in Figure \ref{fig:LQeigenvals}.b.

 This behavior of the exponential stability will affect the convergence of the feedback gain $K_d$, as mentioned in \cite{Morris2020}. Express $K_d=h[\mathbf{k}_q \quad \mathbf{k}_p]$, the behavior of $\mathbf{k}_q$ and $\mathbf{k}_p$ for the FE model for different values of $N$ is shown in Figure \ref{fig:Kn}.a. Note that how when $N$ increase, the values of $\mathbf{k}_q$ and $\mathbf{k}_p$ diverge, i.e., $K_d\begin{bmatrix}
 	\ev^q(t)\\\ev^p(t)
 \end{bmatrix}$ diverges. On the other hand, for the MFEM model, we can see a clear convergence of $\mathbf{k}_q$ and $\mathbf{k}_p$ to some continuous function $k_q(\x)$ and $k_p(\x)$ at system nodes, i.e.,
\begin{align*}
	K_d\begin{bmatrix}
		\ev^q(t)\\\ev^p(t)
	\end{bmatrix}=h\left(\mathbf{k}_q\ev^q(t)+ \mathbf{k}_p \ev^p(t)\right) \quad \to \quad \int_{0}^1 k_q(\x)e^q(\x,t)+k_p(\x)e^p(\x,t) \d\x=K \begin{bmatrix}
	e^q(\x,t)\\e^p(\x,t)
\end{bmatrix}.
\end{align*}

Showing the convergence of $	K_d\begin{bmatrix}
	\ev^q(t)\\\ev^p(t)
\end{bmatrix} $ into $	K\begin{bmatrix}
e^q(\x,t)\\e^p(\x,t)
\end{bmatrix}$, as discussed in \cite{Morris2020,Morris2014}.
 
 \begin{figure}
 	\begin{tabular}{cc}
 		FE model & MFEM model\\
 		\includegraphics[width=0.5\columnwidth]{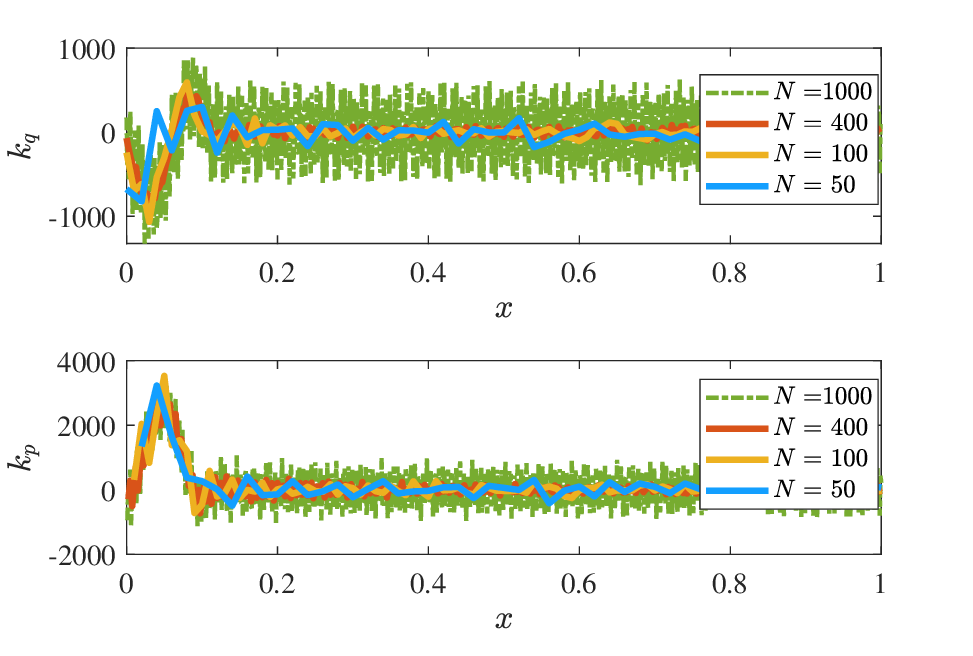}&
 		\includegraphics[width=0.5\columnwidth]{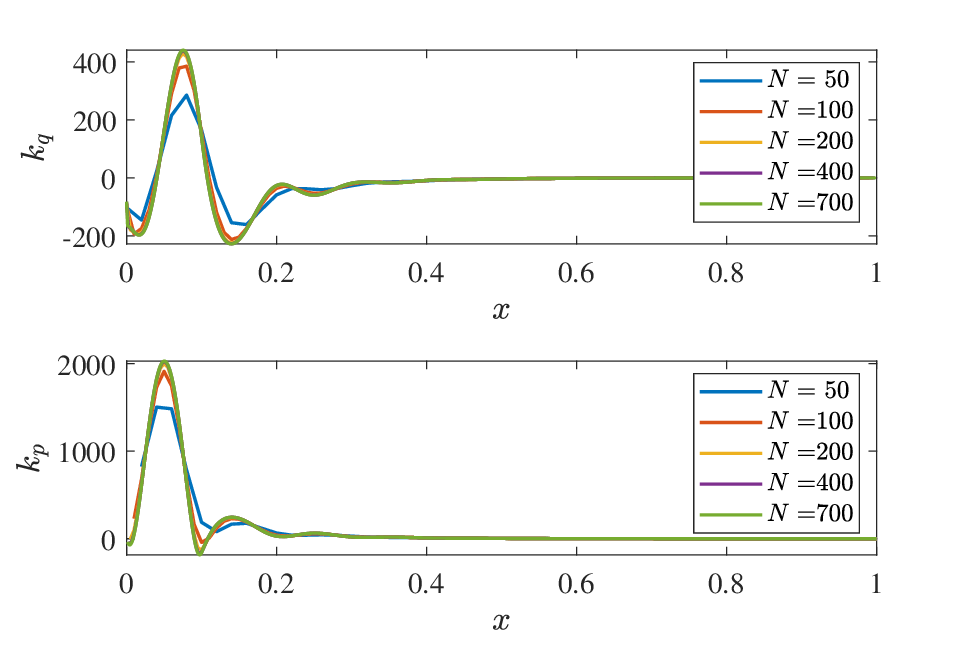}\\
 		(a) & (b)
 	\end{tabular}
 	\caption{ Behavior of  $K_d=h[\mathbf{k}_p \quad \mathbf{k}_q]$ for different values of $N$, considering the FE and MFEM model.}
 	\label{fig:Kn}
 \end{figure}

\section*{Appendix}

In this Section, we present a detailed proof of some Lemmas and Theorems.

\subsection{Proof of Lemma \ref{lem:4}.}

 Let $mspec(X)$ be the multi-spectrum of a square matrix $X$, that is,, the set consisting of the eigenvalues of $X$ including their algebraic multiplicity, and $spec(X)$ be the spectrum of $X$, that is,, the set consisting of the eigenvalues of $X$ ignoring algebraic multiplicity.    
    Since $\Ah=A\otimes I_N$ we obtain that 
    \begin{align*}
        \begin{bmatrix}
            0 & \Ah\\ \Ah^\top & 0
        \end{bmatrix}^{-2}= \begin{bmatrix}
            0 & A\\ A^\top & 0
        \end{bmatrix}^{-2}\otimes I_N= P_1^{-2} \otimes I_N
    \end{align*}

    According to [Bernstain2019, Proposition 9.1.10], $mspec (P_1^{-2}\otimes I_N)=\{\lambda \mu: \lambda\in mspec (P_1^{-2}), \mu \in mspec (I_N)\}$. Note that, in this particular case, $\mu$ always equals $1$. This implies that $spec(P_1^{-2} \otimes I_N)=spec(P_1^{-2})$. Moreover, all $P_1^{-2}$ eigenvalues are real. As a consequence,  $\max eig\left(\begin{bmatrix}
            0 & \Ah\\ \Ah^\top & 0
        \end{bmatrix}^{-2}\right)= \max eig\left(P_1^{-2}\otimes I_N\right)=\max eig(P_1^{-2})$.

    On the other hand, consider the vector $\boldsymbol{f}=[f_1, \dots, f_N]^\top\in\mathbb{R}^N$, the $j$-th element of product $W\boldsymbol{f}$ is 
    $$ \left[W\boldsymbol{f}\right]_j=\begin{cases}
       -\dfrac{h}{4} (f_1+f_2), & j=1\\
        -\dfrac{h}{4} \left((j-1)(f_{j-1}+f_j)+j (f_j +f_{j+1}) \right), &j=2\dots N-1\\
        -\dfrac{h}{4} \left((N-1)(f_{N-1}+f_N)+2N f_N\right), & j=N
    \end{cases}.$$
     
     Using $(a+b)^2\leq a^2+b^2+2|ab|\leq 2a^2+2b^2$ we have
\begin{align*}
    \|W\boldsymbol{f}\|^2=&\dfrac{h^2}{16} \left( (f_1+f_2)^2 + \left((N-1) (f_{N-1}+f_N)+2 N f_N\right)^2+ \sum_{j=2}^{N-1} \left((j-1)(f_{j-1}+f_j) +j(f_j+f_{j+1}) \right)^2\right)\\
    \leq& \dfrac{h^2}{16} \left( (f_1+f_2)^2 +2(N-1)^2(f_{N-1}+f_N)^2+8N^2 f_N^2 +2 
\sum_{j=2}^{N-1} (j-1)^2(f_{j-1}+f_j)^2 +j^2(f_j+f_{j+1})^2  \right).
\end{align*}

Since $\disp \sum_{j=2}^{N-1} (j-1)^2(f_{j-1}+f_j)^2 +j^2(f_j+f_{j+1})^2 
= (f_1+f_2)^2+ (N-1)^2(f_{N-1}+f_N)^2 +2\sum_{j=2}^{N-2} j^2(f_{j}+f_{j+1})^2$, then
\begin{align*}
     \|W\boldsymbol{f}\|^2 \leq & \dfrac{h^2}{16} \left( 3 (f_1+f_2)^2 +4 (N-1)^2(f_{N-1}+f_N)^2 +8 N^2 f_N^2 + 4\sum_{j=2}^{N-2} j^2(f_{j}+f_{j+1})^2 \right)\\
    \leq & \dfrac{h^2}{16} \left( 6 f_1^2+  8\sum_{j=2}^{N} \left(j^2+(j-1)^2\right)f_{j}^2 \right) \\
    \leq&\dfrac{h^2 N^2}{16} \left( 6 f_1^2+  16\sum_{j=2}^{N} f_{j}^2 \right)\\
    \leq&  \ell^2 \|\boldsymbol{f}\|^2.
\end{align*}

Consequently, for any vector $\boldsymbol{g}=[\boldsymbol{g}_1^\top, \dots, \boldsymbol{g}_n^\top]^\top$ with $\boldsymbol{g}_j\in\mathbb{R}^N$, $\forall j\in\{1,\dots,n\}$, we obtain that $\disp\|\Wh\boldsymbol{g}\|^2=\sum_{j=1}^{n} \|W \boldsymbol{g}_j\|^2\leq \ell^2 \sum_{j=1}^{n} \| \boldsymbol{g}_j\|^2= \ell^2 \|\boldsymbol{g}\|^2 $, completing the proof.

\subsection{Proof of Lemma \ref{lem:inner}.}
Denotes by $[V\boldsymbol{f}]_j$ the $j$-th element of the product between a matrix $V$ and vector $\boldsymbol{f}$. Consider matrices $C$, $M$, $D$, and $\tr$ defined at \eqref{eq:MDtr} and \eqref{eq:C}.
Defining $\bpi_k=1/\bp_k$ and $\boldsymbol{f}=[f_1,\dots, f_N]^\top$, we have that $[M\widetilde{Q}^{-1}M^\top \boldsymbol{f}]_1=\dfrac{\bpi_1 f_1}{4}+\dfrac{\bpi_2(f_1+f_2)}{4}$,  $[M\widetilde{Q}^{-1}M^\top \boldsymbol{f}]_j= \dfrac{\bpi_{j} (f_{j-1}+f_j)}{4}  +\dfrac{\bpi_{j+1}(f_j+f_{j+1})}{4}$ for every $j\in\{2,\dots,N-1\}$, and $[M\widetilde{Q}^{-1}M^\top \boldsymbol{f}]_N= \dfrac{\bpi_{N} (f_{N-1}+f_N)}{4} $.  Similarly, $[C\boldsymbol{f}]_1=\dfrac{f_2}{2}$, $[C\boldsymbol{f}]_j=\dfrac{j(f_{j+1}-f_{j-1})}{2}$ for every $j\in\{2,\dots,N-1\}$, and $[C\boldsymbol{f}]_N=N(f_N-f_{N-1})$.
  Then, the inner product $ h\inner{ M \widetilde{\L}^{-1} M^\top \boldsymbol{f},C \boldsymbol{f}}$ is expressed as
    \begin{align*}
    h\inner{ M \widetilde{\L}^{-1} M^\top \boldsymbol{f},C \boldsymbol{f}}=& \dfrac{h}{8} \left(\bpi_{1} f_1 +\bpi_{2}(f_1+f_2)\right)f_2+\dfrac{hN}{4} \bpi_{N}(f_N+f_{N-1})(f_{N}-f_{N-1})\\
    &+\sum_{j=2}^{N-1}\dfrac{h j}{8} \left[\bpi_{j} (f_{j-1}+f_{j})+\bpi_{j+1}(f_j+f_{j+1})\right](f_{j+1}-f_{j-1})
\end{align*}

Expressing $f_{j+1}-f_{j-1}$ as $(f_{j}+f_{j+1})-(f_{j}+f_{j-1})$ we obtain $\disp\sum_{j=2}^{N-1} j \left[\bpi_{j} (f_{j-1}+f_{j})+\bpi_{j+1}(f_j+f_{j+1})\right](f_{j+1}-f_{j-1})=-\sum_{j=2}^{N}\bpi_{j}(f_{j-1}+f_j)^2- \bpi_{2} (f_1+f_2)^2
    +N\bpi_{N}(f_{N-1}+f_N)^2 +\sum_{j=2}^{N-1}j(\bpi_{j}-\bpi_{j+1})(f_{j}+f_{j-1})(f_{j}+f_{j+1})$.
Similarly, note that $ \bpi_{N}(f_N+f_{N-1})(f_{N}-f_{N-1})= 2\bpi_{N}(f_N+f_{N-1})f_N-\bpi_{N}(f_N+f_{N-1})^2$ and $ \left(\bpi_{1} f_1 +\bpi_{2}(f_1+f_2)\right)f_2= -\bpi_{1}f_1^2+(\bpi_{1}-\bpi_{2})(f_1+f_2)f_1+\bpi_{2}(f_1+f_2)^2$. This implies that
\begin{align*}
    h\inner{ M \widetilde{\L}^{-1} M^\top \boldsymbol{f},C \boldsymbol{f}}=& \dfrac{h}{8}\left[-\bpi_{1}f_1^2  -\sum_{j=2}^{N}\bpi_{j}(f_{j-1}+f_j)^2  + 4N\bpi_{N}(f_N+f_{N-1})f_N-N\bpi_{N}(f_{N-1}+f_{N})^2 \right.\\
    & \left. + (\bpi_{1}-\bpi_{2})(f_1+f_2)f_1
    +\sum_{j=2}^{N-1}j(\bpi_{j}-\bpi_{j+1})(f_{j}+f_{j-1})(f_{j}+f_{j+1})\right]\\
    =&-\dfrac{h}{2}\boldsymbol{f}^\top M \widetilde{\L}^{-1} M^\top \boldsymbol{f}+\dfrac{h}{2}\boldsymbol{f}^\top M \widetilde{\L}^{-1} O \widetilde{\L}^{-1}M^\top \boldsymbol{f}\\
    & +\dfrac{h}{8 }\left[4N\bpi_{N}(f_N+f_{N-1})f_N-N\bpi_{N}(f_{N-1}+f_{N})^2 \right],
\end{align*}
since $\bpi_{j}-\bpi_{j+1}=\bpi_j(\bp_{j+1}-\bp_{j})\bpi_{j+1}$.
Finally, using the Young's inequality $4N\bpi_{N}(f_N+f_{N-1})f_N\leq  N\bpi_{N} (f_N+f_{N-1})^2+ 4N\bpi_{N} f_N^2$ we obtain
\begin{align*}
    h\inner{ M \widetilde{\L}^{-1} M^\top \boldsymbol{f},C \boldsymbol{f}}
    =&-\dfrac{h}{2}\boldsymbol{f}^\top M \widetilde{\L}^{-1} M^\top \boldsymbol{f} +\dfrac{h}{2} \boldsymbol{f}^\top M \widetilde{\L}^{-1} O \widetilde{\L}^{-1} M^\top \boldsymbol{f} +\dfrac{h N}{2}\bpi_{N} f_N^2\\
    =&-\dfrac{h}{2} \boldsymbol{f}^\top M \widetilde{\L}^{-1}\left(\widetilde{Q}- O\right) \widetilde{\L}^{-1} M^\top \boldsymbol{f} +\dfrac{h N}{2}\bpi_{N} f_N^2+\dfrac{\ell \bpi_N}{2} \boldsymbol{f}^\top \tr\tr^\top \boldsymbol{f}.
\end{align*}


On the other hand, note that
\begin{align*}
    -\dfrac{1}{h}\inner{D^\top \widetilde{\L} D \boldsymbol{f},C \boldsymbol{f}}=&-\dfrac{1}{2h} \left[ \left(\bp^q_{1}f_1-\bp^q_{2}(f_2-f_1)\right)f_2  +\sum_{j=2}^{N-1} j[\bp_{j}(f_j-f_{j-1})-\bp_{j+1}(f_{j+1}-f_j)](f_{j+1}-f_{j-1})\right]\\
    & -\dfrac{N}{h} \bp_{N}(f_N-f_{N-1})^2.
    \end{align*} 

Given  $\disp\sum_{j=2}^{N-1} j\left[\bp_{j}(f_j-f_{j-1})-\bp_{j+1}(f_{j+1}-f_j)\right](f_{j+1}-f_{j-1}) = \bp_{2}(f_2-f_2)^2 -N\bp_{N}(f_N-f_{N-1})^2+\sum_{i=2}^{N}\bp_{j}\left(f_{j}-f_{j-1}\right)^2 \newline -\sum_{j=2}^{N-1} j(\bp_{j+}-\bp_{j}) (f_j-f_{j-1}) (f_{j+1}-f_j)$ and
$\left(\bp_{1}f_1-\bp_{2}(f_2-f_1)\right)f_2=\bp_{1}f_1^2-(\bp_{2}-\bp_{1})(f_2-f_1)f_1-\bp_{2}(f_2-f_1)^2$

we have that
\begin{align*}
    -\dfrac{1}{h}\inner{D^\top \widetilde{\L} D \boldsymbol{f},C \boldsymbol{f}}=&-\dfrac{1}{2h} \left[ \bp_{1}f_1^2 +\sum_{i=2}^{N}\bp_{j}(f_{j}-f_{j-1})^2  +N \bp_{N}(f_N-f_{N-1})^2 \right.\\
    &\left.-(\bp_{2}-\bp_{1})(f_2-f_1)f_1-\sum_{j=2}^{N-1} j(\bp_{j+1}-\bp_{j}) (f_j-f_{j-1}) (f_{j+1}-f_j)\right]\\
    =&-\dfrac{1}{2h}\boldsymbol{f}^\top D^\top \widetilde{\L} D\boldsymbol{f}+\dfrac{1}{2h}\boldsymbol{f}^\top D^\top O D\boldsymbol{f}-\dfrac{N\bp_{N}}{2h}(f_{N}-f_{N-1})^2\\
    =& -\dfrac{1}{2h}\boldsymbol{f}^\top D^\top \left(\widetilde{\L}-O\right) D\boldsymbol{f}-\dfrac{N\bp_N}{2h}\boldsymbol{f}^\top D^\top \tr\tr^\top D\boldsymbol{f},
\end{align*}
completing the proof.

\bibliographystyle{siam}
\bibliography{references.bib}

\end{document}